\documentclass[12pt]{amsart}
\usepackage[a4paper, total={15cm, 21cm}]{geometry}
\usepackage[T1]{fontenc}
\usepackage{amsmath}
\usepackage{amsfonts}
\usepackage{amssymb}
\usepackage{amsthm}
\usepackage{stmaryrd}
\usepackage{graphicx}
\usepackage{tikz}
\usepackage{tikz-cd}
\usepackage{xcolor}
\usepackage{nicefrac}
\usepackage{hyperref}
\usepackage{enumitem}

\newtheorem{theorem}{Theorem}[section]
\newtheorem{proposition}[theorem]{Proposition}

\newtheorem{corollary}[theorem]{Corollary}
\newtheorem{lemma}[theorem]{Lemma}

\theoremstyle{definition}

\newtheorem{remark}[theorem]{Remark}
\newtheorem{example}[theorem]{Example}

\newcommand{\Ps}{\mathbb{P}}
\newcommand{\C}{\mathbb{C}}
\newcommand{\Z}{\mathbb{Z}}
\newcommand{\R}{\mathbb{R}}
\newcommand{\N}{\mathbb{N}}

\newcommand{\Q}{\mathbb{Q}}

\newcommand{\End}{\mathrm{End}}
\newcommand{\Per}{\mathrm{Per}}
\newcommand{\PrePer}{\mathrm{PrePer}}
\setlist{label=\protect{(\roman*)}, itemsep=5pt}

\title[Centralizers of endomorphisms]{On the centralizers of endomorphisms of the projective line}
\author{Alonso Beaumont} 
\date{November 2025}
\thanks{IRMAR - UMR CNRS 6625, Université de Rennes. E-mail: alonso.beaumont@univ-rennes.fr. Research conducted as part of my PhD project, supported by the European Research Council (ERC Groups of Algebraic Transformations 101053021).}

\calclayout

\begin{document}

\begin{abstract}
    Let $f$ be a dominant endomorphism of the projective line, which is not conjugate to a power map $z\mapsto z^{\pm d}$. We consider the centralizers of the iterates of $f$,
    $C(f^{n}):=\{\textrm{dominant}\;g:\Ps^{1}\rightarrow\Ps^{1}\;|\; g\circ f^{n}=f^{n}\circ g\}$, $n\geq1$,
    and prove that their union is equal to $C(f^{N})$ for some $N\geq1$. This solves a conjecture of F. Pakovich \cite[Conjecture 3.2]{Pak24}. As an application, we obtain a Tits alternative for cancellative semigroups of endomorphisms of the projective line, without an assumption of finite generation, extending the results in \cite{BHPT24}.
\end{abstract}

\maketitle

\section{Introduction}

Let $\End(\Ps^{1})^{+}$ be the set of dominant endomorphisms of the projective line over the complex numbers. It has the structure of a semigroup when endowed with the composition operation, which we denote multiplicatively. We consider centralizer semigroups inside $\End(\Ps^{1})^{+}$: for a fixed $f\in \End(\Ps^{1})^{+}$, set
\[C(f)=\{g\in\End(\Ps^{1})^{+}\;|\; gf=fg\}.\]
In \cite[Theorem 1.2]{Pak20}, F. Pakovich established a finiteness result for these semigroups, which is best stated by first excluding certain maps. An endomorphism $f$ will be called \textit{projective linear} (or simply \textit{linear}) if it is of degree $1$, and \textit{exceptional} if it is conjugate to a power map $z\mapsto z^{\pm d}$, a Chebyshev polynomial $\pm T_{d}$, or a Lattès map (see \cite{Mil06} for a presentation of these maps). Pakovich proves that for any non-linear, non-exceptional endomorphism $f$, the semigroup $C(f)$ is \textit{virtually cyclic}, i.e. there exists a finite set $\{h_{1},\cdots,h_{r}\}\subset \End(\Ps^{1})^{+}$ such that
\[C(f)=\{h_{1},\cdots,h_{r}\}\langle f\rangle:=\{h_{i}f^{j};\;1\leq i\leq r,\;j\geq0\}.\]
One can deduce from this result a well-known theorem of Ritt (\cite{Rit23}, \cite{Ere90}): if two non-linear, non-exceptional endomorphisms $f$ and $g$ commute, there exist integers $m,n\geq1$ such that $f^{m}=g^{n}$.

We refer the reader to \cite{Pak21} for further results on centralizers in $\End(\Ps^{1})^{+}$, and to \cite{CM21, CM23, Pak22b, BHPT24} for related results on semigroups in $\End(\Ps^{1})^{+}$.

Ritt's and Pakovich's theorems are illustrated by the following examples (see also \cite[\S 6]{Pak21}):

\begin{example}
    \label{ex1}
    Let $n\geq1$, and $\sigma:z\mapsto\zeta z$ where $\zeta$ is a primitive $n$-th root of unity. For any $g\in\End(\Ps^{1})^{+}$, the map $f:z\mapsto zg(z^{n})$ commutes with $\sigma$. Ritt's theorem can be easily checked for $f$ and $\sigma f$, since $(\sigma f)^{n}=\sigma^nf^n=f^n$. For a general choice of $g$, we have $C(f)=\{\mathrm{id},\sigma,\cdots,\sigma^{n-1}\}\langle f\rangle$. This is a special instance of a virtually cyclic semigroup, since any two maps of equal degree differ by a linear map.

    In the case of polynomials, these are the only kinds of centralizers that occur. More precisely, following \cite{BE87} and \cite{SS95}, one can show that for any non-linear, non-exceptional polynomial $f$, there exists a non-linear polynomial $f_{0}$ and a linear map $\sigma$ of some finite order $n$ such that $C(f)=\{\mathrm{id},\sigma,\cdots,\sigma^{n-1}\}\langle f_{0}\rangle$.
\end{example}

\begin{example}
    \label{ex2}
    Centralizers of rational maps display a wider range of behaviours than those of polynomial maps. First of all, the linear maps that commute with a given non-linear $f$ may form a non-cyclic finite group. More interestingly perhaps, a centralizer $C(f)$ may have elements of equal degree that do not differ by a linear map, as observed by Ritt \cite{Rit23}. Indeed, consider
    \[f:z\mapsto z\frac{z^{3}-8}{z^{3}+1}.\]
    This is an instance of Example \ref{ex1}; it commutes with $\sigma:z\mapsto\zeta z$ where $\zeta$ is a primitive third root of unity. However, $f$ also has a non-trivial decomposition:
    \[f=uv,\;\mathrm{where}\quad u:z\mapsto\frac{z^{2}-4}{z-1},\quad v:z\mapsto\frac{z^{2}+2}{z+1}.\]
    From the relation $uv\sigma=\sigma uv$ we deduce that $vu$ and $v\sigma u$ commute. More generally, for any $w$ that commutes with $\sigma$, $g:=vwu$ and $h:=v\sigma wu$ commute. For a general choice of $w$, the maps $g$ and $h$ are non-exceptional and $g\neq\tau h$ for any linear $\tau$; Ritt's and Pakovich's results still hold, more precisely $g^{3}=h^{3}$ and $C(g)=\{\mathrm{id},h,h^{2}\}\langle g\rangle$.
\end{example}

In \cite[Conjecture 3.1]{Pak24}, Pakovich asked whether the semigroup
\[C(f{^{\infty}}):=\bigcup_{n\geq1}C(f^{n})\]
has a virtually cyclic structure similar to that of $C(f)$. We answer this question positively:

\begin{theorem}
    \label{theo1}
    For any non-linear, non-exceptional $f\in\End(\Ps^{1})^{+}$ over the complex numbers, $C(f^{\infty})$ is virtually cyclic.
\end{theorem}

Our result also holds for subsemigroups of $C(f^{\infty})$; in particular it provides an alternative proof of \cite[Theorem 1.2]{Pak20}. Note that unlike Pakovich's proof, we do use Ritt's theorem, and we don't obtain an effective bound on the degrees of $h_{1},\cdots,h_{r}$. As remarked by Pakovich, Theorem \ref{theo1} implies that the sequence of centralizers is stationary: there exists an $N\geq 1$ such that $C(f^{\infty})=C(f^{N})$. This is true for a slightly larger class of endomorphisms:

\begin{corollary}
    \label{cor1}
    For any $f\in\End(\Ps^{1})^{+}$ which is not conjugate to a power map, there exists $N\geq1$ such that $C(f^{\infty})=C(f^{N})$.
\end{corollary}

This result is a step forward towards understanding the structure of sets of endomorphisms of $\Ps^{1}$ that share some dynamical data. An example of great interest comes from the notion of the \textit{measure of maximal entropy}. To a non-linear endomorphism $f$ one can attach a canonical invariant measure $\mu_{f}$ whose support is the Julia set of $f$, first defined in \cite{Bro65} for polynomials, and in \cite{FLM83}, \cite{Lju83} for general endomorphisms. This motivates the definition of the following semigroup:
\[E(f):=\{g\in\End(\Ps^{1})^{+}\;|\;g^{*}\mu_{f}=\deg(g)\mu_{f}\}.\]
Equivalently, it is the semigroup of elements of $\End(\Ps^{1})^{+}$ with the same measure of maximal entropy as $f$, to which we  also add the linear maps preserving $\mu_{f}$. We note that $C(f^{\infty})\subset E(f)$ (see Subsection \ref{ss21}), although the inclusion is in general strict. The relations between elements inside $E(f)$ have been studied in \cite{Lev90}, \cite{LP97}, \cite{Ye15}. The structure of this semigroup is only well-understood in the polynomial case, as established in \cite{BE87} and \cite{SS95}: if $f$ is a non-exceptional polynomial, $E(f)$ is virtually cyclic.\\

The proof of Theorem \ref{theo1} is done in two steps. First, we prove that the semigroup $C(f^{\infty})$ has infinitely many finite orbits. This is achieved following an arithmetic equidistribution argument of V. Huguin \cite{Hug23}. We can therefore fix a finite orbit $\mathcal{O}$ of size at least $3$, which can be chosen inside the common Julia set of the non-linear elements of $C(f^{\infty})$. The second step consists in proving that an element of  $C(f^{\infty})$ is characterized by its degree and its action on $\mathcal{O}$; the virtual cyclicity of $C(f^{\infty})$ follows from this. This second step is obtained following algebraic arguments which appear in the orginial proof of Ritt's theorem \cite{Rit23}.\\

The proof strategy is insipired by \cite{BHPT24}, which establishes an analog of the Tits alternative for finitely generated semigroups of $\End(\Ps^{1})^{+}$. In turn, our result allows us to strengthen this alternative by dropping the assumption of finite generation.

A semigroup $S$ is said to be \textit{left cancellative} (resp. \textit{right cancellative}) if for any $a,b,c\in S$, the relation $ab=ac$ (resp. $ba=ca$) implies $b=c$; a semigroup is \textit{cancellative} if it is both left and right cancellative. Note that $\End(\Ps^{1})^{+}$ (and any of its subsemigroups) is right cancellative, as it is comprised of dominant maps. We say that a semigroup $S$ satifies an \textit{identity} if there exist two distinct words $w_{1},w_{2}$ on two letters such that for any $a,b\in S$, $w_{1}(a,b)=w_{2}(a,b)$ as elements of $S$. We begin by stating a result over a finitely generated field; it is an analog of the generalized Tits alternative of J. Okniński and A. Salwa \cite{OS95}.

\begin{theorem}
    \label{theo2}
    Let $S$ be a semigroup in $\End(\Ps^{1})^{+}$ defined over a finitely generated field of characteristic $0$. The following assertions are equivalent:
    \begin{enumerate}
        \item $S$ does not contain a free subsemigroup of rank $2$.
        \item $S$ satisfies a semigroup identity.
        \item Any cancellative subsemigroup of $S$ can be embedded in a virtually abelian group.
    \end{enumerate}
\end{theorem}

A semigroup $S$ in $\End(\Ps^{1})^{+}$ will be called \textit{linear} if it consists of linear maps, and it will be called a \textit{power semigroup} if, up to simultaneous conjugation, its elements are of the form $z\mapsto\zeta z^{\pm d}$, where $\zeta$ is a root of unity and $d\geq1$. The following is a counter-example to the theorem above if we drop the assumption of a finitely generated field of definition:

\begin{example}
    Let $S$ be the power semigroup consisting of elements of the form $p_{d,\zeta}:z\mapsto \zeta z^{d}$ where $d$ is a power of three and $\zeta$ is a root of unity of order a power of two. The linear elements in $S$ form a group which is isomorphic to the Prüfer $2$-group $\Z(2^{^{\infty}}):=\Z[1/2]/\Z$. Moreover, the element $p_{3,1}$ acts on $\Z(2^{^{\infty}})$ by multiplication by $3$. This defines an embedding
    \[S \longrightarrow G:=\Z(2^{^{\infty}})\rtimes\Z.\]
    This embedding is the smallest possible for $S$; any other embedding of $S$ into a group factors through it. Now, $G$ is not virtually abelian. Indeed, if $H$ is a finite index subgroup of $G$, then $H\cap \Z(2^{^{\infty}})$ is of finite index in $\Z(2^{^{\infty}})$, so $H\cap \Z(2^{^{\infty}})=\Z(2^{^{\infty}})$. Moreover, $H$ contains some $p_{3^{n},1}$. If we let $p_{1,\zeta}\in H$ be of large enough order, then $p_{1,\zeta}p_{3^{n},1}=p_{3^{n},\zeta}\neq p_{3^{n},\zeta^{3^{n}}}=p_{3^{n},1}p_{1,\zeta}$, so $H$ is not abelian. On the other hand, $G$ is \textit{locally virtually abelian}. Indeed, for any finitely generated $H$ in $G$, $H\cap \Z(2^{^{\infty}})$ is finite, so $H$ is finite-by-abelian and finitely generated and thus virtually abelian. In particular, $S$ cannot contain two elements that generate a free subsemigroup of rank $2$.
\end{example}

One can also construct linear counter-examples in the same vein, such as the semigroup $\{z\mapsto\zeta z+t\;;\;\zeta\textrm{ root of unity },\;t\in\C\}$. Outside of examples like these ones, any cancellative semigroup in $\End(\Ps^{1})^{+}$ will satisfy Theorem \ref{theo2}, without an initial assumption of finite generation. More precisely, we have the following

\begin{corollary}
    \label{cor2}
    Let $S$ be a cancellative semigroup in $\End(\Ps^{1})^{+}$ defined over the complex numbers. Then one of the following holds:
    \begin{enumerate}
        \item $S$ contains a free subsemigroup of rank $2$.
        \item $S$ can be embedded in a virtually abelian group, unless it is a linear or power semigroup, in which case it can be embedded in a locally virtually abelian group.
    \end{enumerate}
\end{corollary}

\section{Preliminaries}\label{s2}

\subsection{Measure of maximal entropy}\label{ss21} Let $f\in\End(\Ps^{1})^{+}$. Recall that the pullback of a measure $\mu$ by $f$ is defined by $f^{*}\mu:A\mapsto\int_{\Ps^{1}(\C)}\#(f^{-1}(z)\cap A)d\mu(z)$, where each set $f^{-1}(z)\cap A$ is counted with multiplicity.

Assume that $f$ is non-linear. Its exceptional set is defined as the largest finite set in $\Ps^{1}(\C)$ that is invariant by $f^{-1}$. The \textit{measure of maximal entropy} of $f$ can be characterized as the unique probability measure $\mu_{f}$ such that $f^{*}\mu_{f}=\deg(f)\mu_{f}$, whose support is not contained in the exceptional set (see \cite{FLM83}, \cite{Lju83}). For any non-linear $g\in E(f)$, we have $g^{*}\mu_{f}=\deg(g)\mu_{f}$, so by the uniqueness of $\mu_{g}$, we have $\mu_{f}=\mu_{g}$, as stated in the introduction. In particular, $\mu_{f^{n}}=\mu_{f}$ for any $n\geq1$.

Let us also briefly explain why $C(f^{\infty})\subset E(f)$. If $g$ commutes with $f$, then the probability measure $\frac{1}{\deg(g)}g^{*}\mu_{f}$ satisfies
\[f^{*}\left(\frac{1}{\deg(g)}g^{*}\mu_{f}\right)=\frac{1}{\deg(g)}g^{*}(f^{*}\mu_{f})=\deg(f)\left(\frac{1}{\deg(g)}g^{*}\mu_{f}\right)\]
so by the uniqueness of $\mu_{f}$, we have $\frac{1}{\deg(g)}g^{*}\mu_{f}=\mu_{f}$, that is, $g\in E(f)$. For any $n\geq1$, we have $\mu_{f^{n}}=\mu_{f}$, so $g\in E(f)$ for any $g\in C(f^{n})$, hence for any $g\in C(f^{\infty})$. In particular, a non-linear element of $C(f^{\infty})$ has the same Julia set as $f$.

\subsection{Characterizations of \texorpdfstring{$C(f^{\infty})$}{C(f∞)}}\label{ss22} For any endomorphism $f\in \End(\Ps^{1})^{+}$, we denote by $\Per(f)$ its set of periodic points, and by $\PrePer(f)$ its set of preperiodic points.

We start by stating a crucial lemma that will be used throughout this text (see \cite[Lemma 2]{Lev90}, \cite[Theorem 1.5]{Ye15}):

\begin{lemma}
    \label{lem1}
    Let $f,g\in \End(\Ps^{1})^{+}$ be two non-linear endomorphisms with the same measure of maximal entropy. Suppose they have a common fixed point in their common Julia set. Then $f$ and $g$ commute.
\end{lemma}

The following proposition gives two different characterizations of $C(f^{\infty})$. It is essentially a reformulation of \cite[Theorem 1.5]{Ye15}. Its main ingredients are Ritt's theorem \cite{Rit23}, Lemma \ref{lem1} above, and a deep rigidity result of X. Yuan and S.-W. Zhang \cite[Theorem 1.4]{YZ21a}:

\begin{proposition}
    \label{prop1}
    Let $f\in\End(\Ps^{1})^{+}$ be a non-linear, non-exceptional endomorphism, and let $g\in\End(\Ps^{1})^{+}$ be a non-linear endomorphism. The following are equivalent:
    \begin{enumerate}
        \item $g\in C(f^{\infty})$.
        \item There exist $m,n\geq1$ such that $g^{m}=f^{n}$.
        \item $\Per(g)=\Per(f)$.
    \end{enumerate}
    Moreover, a linear map $\sigma$ is contained in $C(f^{\infty})$ if and only if $\sigma(\Per(f))=\Per(f)$.
\end{proposition}

\begin{proof}
    If $g$ commutes with $f^{n}$ for some $n\geq1$, then by Ritt's theorem \cite{Rit23} there exist $\ell,m$ such that $g^{\ell}=f^{mn}$: this establishes $(i)\Rightarrow(ii)$. If $g^{m}=f^{n}$ for some integers $m,n\geq1$, then we have the following equality of periodic sets: $\Per(g)=\Per(g^{m})=\Per(f^{n})=\Per(f)$. This establishes $(ii)\Rightarrow(iii)$. If $\Per(g)=\Per(f)$, then some iterates $g^{m}$ and $f^{n}$ have a common fixed point in their common Julia set. By \cite[Theorem 1.4]{YZ21a}, $g^{m}$ and $f^{n}$ also have a common measure of maximal entropy so by Lemma \ref{lem1} they commute, and hence $g\in C(f^{\infty})$. This establishes $(iii)\Rightarrow(i)$. 

    If a linear map $\sigma$ is in $C(f^{\infty})$ then $\sigma f$ is in $C(f^{\infty})$, so by $(i)\Rightarrow(iii)$, we have $\Per(\sigma f)=\Per(f)$. Therefore, $\sigma(\Per(f))=(\sigma f)(\Per(f))=\Per(f)$. Conversely, suppose that $\sigma(\Per(f))=\Per(f)$. We have $f^{m}(z)=z$ for some $m$ if and only if $f^{n}(\sigma(z))=\sigma(z)$ for some $n$, and thus $\Per(\sigma^{-1}f\sigma)=\Per(f)$. By $(iii)\Rightarrow(ii)$ there exists some $\ell$ such that $\sigma^{-1}f^{\ell}\sigma=f^{\ell}$, which means that $\sigma\in C(f^{\infty})$.
\end{proof}

\subsection{Exceptional maps}\label{ss23} A \textit{power map} is an endomorphism of the form $z\mapsto z^{\pm d}$ where $d\geq1$. For convenience we will also refer to any endomorphism of the form $z\mapsto \zeta z^{\pm d}$, where $\zeta$ is a root of unity, as a power map. \textit{Chebyshev polynomials} are defined by induction as follows:
\[T_{1}(z)=z,\quad T_{2}(z)=z^{2}-2, \quad T_{d+1}(z)=zT_{d}(z)-T_{d-1}(z).\]
Equivalently, $T_{d}$ is the unique polynomial such that $T_{d}(z+z^{-1})=z^{d}+z^{-d}$. For convenience we will also refer to the family $(-T_{d})_{d\geq1}$ as Chebyshev polynomials. Note that the $T_{d}$ are defined over the integers.\\

An endomorphism $f\in\End(\Ps^{1})^{+}$ is called a \textit{Lattès map} if there exists an elliptic curve $E$ and dominant morphisms $\pi:E\rightarrow\Ps^{1}$, $\varphi:E\rightarrow E$ such that $\pi\varphi=f\pi$. By \cite[Theorem 3.1]{Mil06}, for a given Lattès map $f$, we may choose $\pi$ to be a quotient map $E\rightarrow E/\Gamma\cong\Ps^{1}$ where $\Gamma$ is a cyclic group of automorphisms of $E$ of order $2,3,4$, or $6$. Moreover, $\varphi$ can be written in affine form $z\mapsto [\alpha](z)+t$ where $\alpha$ is an integer or quadratic integer in the endomorphism ring of $E$, and $t\in E$. If we fix a generator $[\omega]$ of $\Gamma$, the relation $\pi\varphi=f\pi$ implies that
\[\forall z\in E,\quad [\alpha]([\omega](z))+t=[\omega^{k}]([\alpha](z)+t)\]
for some integer $k$. Comparing differentials on both sides we see that $k=1$ and thus $t=[\omega](t)$: the translation part of a Lattès map is necessarily a $[1-\omega]$-torsion point.

Fix an elliptic curve $E$ and a group of automorphisms $\Gamma=\langle[\omega]\rangle$ of $E$. Denote by $R$ the ring of endomorphisms of $E$ and by $E[1-\omega]$ its set of $[1-\omega]$-torsion points. We define $L(E,\Gamma)$ as the semigroup of Lattès maps $\{\ell_{\alpha,t}\;;\;\alpha\in R\backslash\{0\},\;t\in E[1-\omega]\}$, where $\ell_{\alpha,t}$ fits into the following diagram:
\begin{equation}
    \label{eq1}
    \begin{tikzcd}
	E &&& E \\
	{E/\Gamma} & {\Ps^{1}} & {\Ps^{1}} & {E/\Gamma}
	\arrow["{z\mapsto[\alpha](z)+t}", from=1-1, to=1-4]
	\arrow[from=1-1, to=2-1]
	\arrow[from=1-4, to=2-4]
	\arrow["\cong"{description}, draw=none, from=2-1, to=2-2]
	\arrow["{\ell_{\alpha,t}}", from=2-2, to=2-3]
	\arrow["\cong"{description}, draw=none, from=2-3, to=2-4]
    \end{tikzcd}
\end{equation}
The semigroup $L(E,\Gamma)$ is well defined up to the choice of isomorphism $E/\Gamma\cong\Ps^{1}$, that is, up to simultaneous conjugation by an automorphism of $\Ps^{1}$.

\begin{remark}
    The term Lattès map usually refers to an endomorphism of degree at least two. By abuse of terminology, we will also call Lattès map any linear endomorphism that fits into diagram \eqref{eq1}.
\end{remark}

We will introduce some auxiliary notation for the following lemma. For a non-linear $f\in\End(\Ps^{1})^{+}$, we define
\[P(f)=\{g\in\End(\Ps^{1})^{+}\;|\; g(\PrePer(f))=\PrePer(f)\}.\]
By \cite[Theorem 1.3]{YZ21a}, $P(f)$ is the set of endomorphisms with the same set of preperiodic points as $f$, to which we also add the linear maps preserving $\PrePer(f)$. In particular, the elements $g\in P(f)$ satisfy the seemingly stronger property 
\[g^{-1}(\PrePer(f))=\PrePer(f).\]
If an endomorphism $g$ commutes with $f$, then it sends finite $f$-orbits to finite $f$-orbits, so $g\in P(f)$. For any $n\geq1$, $\PrePer(f^{n})=\PrePer(f)$. From these two observations we conclude that $C(f^{\infty})\subset P(f)$. Finally, we remark that $P(f)$ coincides with the semigroup $E(f)$ defined in the introduction, except when $f$ is conjugate to a power map. This fact won't be needed in our proofs.

\begin{lemma}
    \label{lem2}
    Let $f\in\End(\Ps^{1})^{+}$ be a non-linear, exceptional endomorphism.
    \begin{enumerate}
        \item If $f$ is conjugate to a power map, then all elements in $P(f)$ are simultaneously conjugate to power maps.
        \item If $f$ is conjugate to a Chebyshev polynomial, then all elements in $P(f)$ are simultaneously conjugate to Chebyshev polynomials.
        \item If $f\in L(E,\Gamma)$ for some elliptic curve $E$ and group of automorphisms $\Gamma$, then $P(f)=L(E,\Gamma)$.
    \end{enumerate}
\end{lemma}

\begin{proof}
    Suppose that, up to conjugation, $f$ is a power map, and let $g\in P(f)$. The Julia set $J$ of $f$ is the set of accumulation points of $\PrePer(f)$, so $g^{-1}(J)=J$. Moreover, $\PrePer(f)\backslash J=\{0,\infty\}$, and thus $g^{-1}(\{0,\infty\})=\{0,\infty\}$. We deduce that $g$ is of the form $z\mapsto\zeta z^{\pm d}$ where $\zeta\in\C^{\times}$, $d\geq1$. Finally,
    \[\zeta=g(1)\in g(\PrePer(f))=\PrePer(f)\]
    so $\zeta$ is a root of unity, and hence $g$ is a power map.
    
    Suppose that, up to conjugation, $f$ is a Chebyshev polynomial, and let $g\in P(f)$. As before, denoting by $J=[-2,2]$ the Julia set of $f$, we have $g^{-1}(J)=J$, and since $\PrePer(f)\backslash J=\{\infty\}$, we have $g^{-1}(\{\infty\})=\{\infty\}$. In particular, $g$ is a polynomial and $g^{-1}([-2,2])=[-2,2]$: by \cite[Theorem 1.4.1]{Bea91}, $g$ is a Chebyshev polynomial.
    
    Suppose that $f\in L(E,\Gamma)$ and let $g\in P(f)$. Then $g(\PrePer(f))=\PrePer(f)$ so we may apply \cite[Theorem 27]{KS07}, \cite[Lemma 4.9]{BHPT24} to conclude that $g$ is a Lattès map fitting into the same diagram \eqref{eq1} as $f$, that is, $g\in L(E,\Gamma)$. Note that \cite[Theorem 27]{KS07} is stated for endomorphisms $g$ of degree at least $2$, but the proof holds for linear endomorphisms without modification. Conversely, let $g\in L(E,\Gamma)$ and denote its lift by $\varphi:E\rightarrow E$. The set $E_{\mathrm{tors}}$ of torsion points of $E$ is fully invariant by $\varphi$ (recall that the translation part of $\varphi$ is a torsion point) so $E_{\mathrm{tors}}/\Gamma$ is fully invariant by $g$. Since $E_{\mathrm{tors}}/\Gamma=\PrePer(f)$ (\cite[Corollary 31]{KS07}) we obtain $g\in P(f)$.
\end{proof}

\section{Finite orbits of \texorpdfstring{$C(f^{\infty})$}{C(f∞)}}

The purpose of this section is to prove the following proposition:
\begin{proposition}
    \label{prop2}
    Let $f$ be a non-linear, non-exceptional endomorphism of $\Ps^{1}$ defined over $\C$. There exist infinitely many finite $C(f^{\infty})$-orbits in $\Per(f)$.
\end{proposition}

We begin by proving the assertion under the assumption that $f$ is defined over a number field (Proposition \ref{prop3} $(ii)$), and then use a specialization argument to prove it in the general case (Proposition \ref{prop4} $(ii)$).

Throughout this section, $f$ will denote a non-exceptional endomorphism of $\Ps^{1}$ of degree $d\geq2$.

\subsection{The Lyapunov exponent} We define the \textit{Lyapunov exponent} of $f$ as
\[L_{f}:=\int_{\Ps^{1}(\C)}\log\|f'\|\;d\mu_{f}\]
where $\mu_{f}$ is the measure of maximal entropy of $f$ and $\|f'\|$ denotes the norm of the differential of $f$ with respect to the round metric. We have $L_{f}>0$ as observed in \cite[Lemma  II.3]{Mañ83}

Let $P$ be a periodic cycle of $f$, of size $p$ and multiplier $\lambda$ (i.e. $\lambda=(f^{p})'(z)$ for any $z\in P$). Let $\mu_{P}$ be the uniform measure on $P$. We define the \textit{characteristic exponent} of $f$ along $P$ as
\[\chi_{f}(P):=\int_{\Ps^{1}(\C)}\log\|f'\|\;d\mu_{P}=\frac{1}{p}\log|\lambda|.\]

We will make use of the following criterion for exceptionality:

\begin{theorem}[{\cite[Proposition 4]{Zdu14}, \cite[Remark 10]{Hug23}}]
    \label{theo3}
    If $f$ is a non-linear, non-exceptional endomorphism of $\Ps^{1}$, then the inequality $\chi_{f}(P)> L_{f}$ holds for infinitely many periodic cycles $P$.
\end{theorem}

\subsection{Arithmetic equidistribution} Assume that $f$ is defined over a number field $K\subset \C$. In particular, the set $\PrePer(f)$ is defined over the algebraic closure $\bar{K}$ of $K$ in $\C$. We denote by $G_{K}$ the absolute Galois group $\mathrm{Gal}(\bar{K}/K)$. The proof of Proposition \ref{prop3} will hinge upon the following equidistribution result:

\begin{theorem}
    \label{theo4}
    Let $(E_{k})_{k\geq0}$ be a sequence of finite $G_{K}$-invariant sets contained in $\PrePer(f)$, whose union is infinite. For every $k\geq0$, let $\mu_{k}$ be the uniform measure on $E_{k}$. Then the sequence $(\mu_{k})_{k\geq0}$ converges weakly towards the measure of maximal entropy $\mu_{f}$: for every continuous function $\psi:\Ps^{1}(\C)\rightarrow\R$, we have
    \[\int_{\Ps^{1}(\C)}\psi\; d\mu_{k}\underset{k\rightarrow\infty}{\longrightarrow}\int_{\Ps^{1}(\C)}\psi\; d\mu_{f}.\]
\end{theorem}

Theorem \ref{theo4} was proven by P. Autissier in \cite{Aut01} in a more general setting; see also \cite[Chapter 10]{BR10} for a statement of Autissier's result in the dynamical setting.

\begin{proposition}
    \label{prop3}
    Let $f$ be a non-linear, non-exceptional endomorphism defined over a number field.
    \begin{enumerate}
        \item Let $z\in\Per(f)$ and denote by $P$ its cycle. If the orbit $\mathcal{O}:=C(f^{\infty})\cdot z$ is infinite, then $\chi_{f}(P)\leq L_{f}$.
        \item There exist infinitely many finite $C(f^{\infty})$-orbits in $\Per(f)$.
    \end{enumerate}
\end{proposition}

\begin{proof}
    Let $z\in \Per(f)$. Denote by $P$ its cycle, $p$ its period and $\lambda$ its multiplier. Suppose that $\mathcal{O}:=C(f^{\infty})\cdot z$ is infinite, and let us show that $\chi_{f}(P)\leq L_{f}$. We may clearly assume that $\lambda\neq0$.
    
    \underline{Step 1}. Every element of $\mathcal{O}$ is of the form $g(z)$, where $g\in C(f^{\infty})$. By definition, there exists $n\geq1$ such that $f^{n}g=gf^{n}$, and thus
    \[f^{np}(g(z))=g(f^{np}(z))=g(z)\]
    so $g(z)$ is a periodic point for $f$. Denote by  $P_{g}$ its cycle, $p_{g}|np$ its period and $\lambda_{g}$ its multiplier. The relation $f^{np}g=gf^{np}$ ensures that $z$ and $g(z)$ have the same multiplier \textit{as fixed points of} $f^{np}$. Indeed, we have
    \[(f^{np})'(g(z))\cdot g'(z)=g'(f^{np}(z))\cdot (f^{np})'(z)=g'(z)\cdot (f^{np})'(z)\]
    so $(f^{np})'(z)=(f^{np})'(g(z))$. We use here the fact that $z$ is not a critical point of $g$. This is clear if $g$ is linear, and if $g$ is non-linear, then Proposition \ref{prop1}, $(i)\Rightarrow(ii)$ ensures that $g$ and $f$ have an iterate in common. If $z$ were a critical point of $g$, then it would be a critical point of some iterate of $f^{np}$ yielding $\lambda=0$, a possibility we have excluded. We have thus obtained the equality
    \begin{equation}
        \label{eq2}
        \lambda_{g}^{np/p_{g}}=\lambda^{n}.
    \end{equation}
    In particular,
    \[\chi_{f}(P_{g})=\frac{1}{np}\log|\lambda_{g}^{np/p_{g}}|=\frac{1}{np}\log|\lambda^{n}|=\chi_{f}(P).\]
    
    \underline{Step 2}. Let $K$ be a number field over which $f$ and $\lambda$ are defined, and let $G_{K}$ denote the Galois group $\mathrm{Gal}(\bar{K}/K)$. Since $f$ is defined over $K$, for any $\sigma\in G_{K}$, $\sigma(g(z))$ is periodic for $f$ of period $p_{g}$. Moreover, the multiplier of $\sigma(g(z))$ as a fixed point of $f^{np}$ is $\sigma(\lambda^{n})=\lambda^{n}$, by \eqref{eq2}. In particular, $\chi_{f}(\sigma(P_{g}))=\chi_{f}(P)$. Write
    \[G_{K}\cdot\mathcal{O}=\bigcup_{k\geq0}E_{k}\]
    where each $E_{k}$ is a finite $G_{K}$-invariant set of the form $\{\sigma_{1}(P_{g}),\cdots,\sigma_{r}(P_{g})\}$, where $g\in C(f^{\infty})$ and $\sigma_{i}\in G_{K}$. Denote by $\mu_{k}$ the uniform measure on $E_{k}$. We have
    \begin{equation}
        \label{eq3}
        \int_{\Ps^{1}(\C)}\log\|f'\|\;d\mu_{k}=\frac{1}{r}\sum_{i=1}^{r}\chi_{f}(\sigma_{i}(P_{g}))=\chi_{f}(P).
    \end{equation}
    By Theorem \ref{theo4}, the sequence $(\mu_{k})_{k\geq0}$ converges weakly to $\mu_{f}$. For $t\in\R$, the function $\max(\log\|f'\|,t)$ is continuous, and so in view of \eqref{eq3} we obtain
    \[\chi_{f}(P)\leq\int_{\Ps^{1}(\C)}\max(\log\|f'\|,t)\;d\mu_{k}\underset{k\rightarrow\infty}{\longrightarrow}\int_{\Ps^{1}(\C)}\max(\log\|f'\|,t)\;d\mu_{f}.\]
    If we let $t\rightarrow-\infty$, the monotone convergence theorem ensures that $\chi_{f}(P)\leq L_{f}$ and assertion $(i)$ is proven.

    Suppose there are only finitely many finite $C(f^{\infty})$-orbits in $\Per(f)$. Then by $(i)$ we would have $\chi_{f}(P)\leq L_{f}$ for all but finitely many cycles of $f$, and so $f$ would be exceptional as a consequence of Theorem \ref{theo3}, contradicting our assumption. This proves assertion $(ii)$.
\end{proof}

\subsection{The exceptional locus} See \cite{Sil98} for details on the following definitions. The space of endomorphisms of $\Ps^{1}$ of degree $d$ can be endowed with the structure of an affine variety defined over $\Q$, which we denote by $\mathrm{Rat}_{d}$. Similarly, the conjugacy classes of endomorphisms of degree $d$ are parametrized by an affine variety $\mathrm{M}_{d}$ defined over $\Q$, which is equipped with a quotient map $\pi:\mathrm{Rat}_{d}\rightarrow\mathrm{M}_{d}$.

\begin{lemma}
    \label{lem3}
    The exceptional maps in $\mathrm{Rat}_{d}$ constitute a closed subvariety defined over $\Q$.
\end{lemma}

\begin{proof}
    We begin by describing the exceptional maps modulo conjugation as a subset $[\mathcal{E}]$ of $\mathrm{M}_{d}$. First of all, we have two classes of power maps, $[z\mapsto z^{d}]$ and $[z\mapsto z^{-d}]$, and one or two classes of Chebyshev polynomials $[T_{d}]$, $[-T_{d}]$, depending on whether $d$ is even or odd. We then have the classes of \textit{flexible Lattès maps}, that is, the maps $\ell_{\alpha,t}\in L(E,\Gamma)$ where $\Gamma$ is of order $2$ and $\alpha$ is an integer. These classes make up zero, one or two curves in $\mathrm{M}_{d}$, depending on whether $d$ is a non-square, an even square, or an odd square, and they are all defined over $\Q$ (see \cite[§5]{Mil06}). Finally, the remaining Lattès maps are contained in finitely many classes, as a result of \cite[Lemma 5.4]{Mil06}. These classes are all defined over an algebraic closure $\bar{\Q}$ of $\Q$, and they are permuted by the action of $\mathrm{Gal}(\bar{\Q}/\Q)$.

    From this description it is clear that $[\mathcal{E}]$ is a closed (reducible) subvariety of $\mathrm{M}_{d}$ defined over $\Q$, and therefore the same holds for $\pi^{-1}([\mathcal{E}])\subset\mathrm{Rat}_{d}$.
\end{proof}

We call $\mathcal{E}:=\pi^{-1}([\mathcal{E}])$ the \textit{exceptional locus} within $\mathrm{Rat}_{d}$.

\subsection{Specialization} We now generalize Proposition \ref{prop3} $(ii)$ to endomorphisms defined over arbitrary fields of characteristic $0$, using a specialization argument adapted from \cite[§4]{DFR23}, \cite[§5]{DF17}.

The argument goes roughly as follows: an endomorphism $f$ of $\Ps^{1}$ is defined over a finitely generated field, and hence it is defined over a function field $\bar{\Q}(B)$, where $B$ is a variety over an algebraic closure $\bar{\Q}$ of $\Q$. For any point $b\in B(\bar{\Q})$, we can ``evaluate'' the coefficients of $f$ at $b$ in order to obtain a new endomorphism $f_{b}$, this time defined over a number field. We then show that infinite orbits $C(f^{\infty})\cdot z$ remain infinite after this specialization process in order to apply the arithmetic equidistribution theorem in the proof of Proposition \ref{prop3} $(i)$. Then, the closedness of $\mathcal{E}$ will tell us that if every $f_{b}$ is exceptional then $f$ is exceptional, reaching the same contradiction as in the proof of Proposition \ref{prop3} $(ii)$.

One added difficulty is the fact that we apply the equidistribution result many times, each time choosing a larger number field extension as base field. We will therefore need to make sense of specializations that are \textit{above} a certain point $b\in B(\bar{\Q})$.\\

\paragraph{\texorpdfstring{\textbf{Family.}}{Family}} We return to the general case where $f$ is an endomorphism of $\Ps^{1}$ of degree $d$ defined over the complex numbers. Let $R$ be the ring generated by the coefficients of $f$ and the inverse of the resultant $\mathrm{Res}(P,Q)$ where $f:z\mapsto P(z)/Q(z)$, written in reduced form. Let $B:=\mathrm{Spec}(R\otimes\bar{\Q})$, an affine variety defined over an algebraic closure $\bar{\Q}$ of $\Q$. Then $f$ is defined over the function field $\bar{\Q}(B)$, and it extends to an endomorphism $f_{B}:\Ps^{1}_{B}\rightarrow\Ps^{1}_{B}$ of degree $d$ defined over $\bar{\Q}$. Here, $\Ps^{1}_{B}$ is the projective line over $B$. Denote by $s:\Ps^{1}_{B}\rightarrow B$ its structure morphism. For $b\in B$, we write $\Ps^{1}_{b}=s^{-1}(b)$ and define the \textit{specialization} $f_{b}:=(f_{B})_{|\Ps^{1}_{b}}$ of $f$ at $b$, an endomorphism of $\Ps^{1}_{b}$ defined over $\bar{\Q}(b)$, which is of degree $d$ since $\mathrm{Res}(P,Q)$ is invertible in $R\otimes\bar{\Q}$.

In particular, $f$ induces a family $(f_{b})_{b\in B}$ contained in $\mathrm{Rat}_{d}$ parametrized by $B$, whose generic point is $f$.\\

\paragraph{\texorpdfstring{\textbf{Periodic Points.}}{Periodic points}} For any $n\geq1$, we define $\Per_{n}(f)=\{z\in\Ps^{1}\;|\;f^{n}(z)=z\}$, a subvariety of $\Ps^{1}$ defined over $\bar{\Q}(B)$. Let $\Per_{n}(f)_{B}$ be its Zariski closure in $\Ps^{1}_{B}$, and for any $b\in B$, let $\Per_{n}(f)_{b}=\Per_{n}(f)_{B}\cap\Ps^{1}_{b}$. By construction, the elements of $\Per_{n}(f)_{b}$ satisfy $f_{b}^{n}(z)=z$. Let $\mathrm{mult}_{z}(\Per_{n}(f)_{b})$ be the intersection-theoretic multiplicity of a point $z$ in $\Per_{n}(f)_{b}$. The map
\[b\in B\mapsto \sum_{z\in\Per_{n}(f)_{b}}\mathrm{mult}_{z}(\Per_{n}(f)_{b})\in \N\]
is upper semi-continuous, as a consequence of the properness of $\Per_{n}(f)_{B}\rightarrow B$ and Nakayama's lemma (see \cite[Lemma 5.3]{DF17}). In particular, evaluating this map at the generic point we obtain
\[\sum_{z\in\Per_{n}(f)}\mathrm{mult}_{z}(\Per_{n}(f))\leq\sum_{z\in\Per_{n}(f)_{b}}\mathrm{mult}_{z}(\Per_{n}(f)_{b})\]
for any $b\in B$. The set $\Per_{n}(f)$ has $d^{n}+1$ points counted with multiplicity, and so $\Per_{n}(f)_{b}$ has at least that many points. But the same count holds for $\Per_{n}(f_{b})$ so we obtain the following: \textit{for any $b\in B$, $\Per_{n}(f)_{b}=\Per_{n}(f_{b})$}.\\

\paragraph{\texorpdfstring{\textbf{Base change.}}{Base change}} Let $B'$ be an affine variety over $\bar{\Q}$ equipped with a finite morphism $B'\rightarrow B$. This morphism induces a map $\pi_{B'}:\Ps^{1}_{B'}\rightarrow \Ps^{1}_{B}$ which restricts to an isomorphism $\pi_{b'}:\Ps^{1}_{b'}\rightarrow \Ps^{1}_{b}$ for any point $b'\in B'$ that is above $b\in B$. The endomorphism $f$ is also defined over $\bar{\Q}(B')$ and it extends to a map $f_{B'}:\Ps^{1}_{B'}\rightarrow \Ps^{1}_{B'}$ which satisfies $\pi_{B'}f_{B'}=f_{B}\pi_{B'}$.

Let $z\in \Per(f)$ and $b\in B$. The point $z$ is defined over a finite extension of $\bar{\Q}(B)$. Therefore, there exists some $B'$ equipped with a finite morphism $B'\rightarrow B$ such that $z$ is defined over $\bar{\Q}(B')$. For any $b'\in B'$ that is above $b$, we can define the specialization $z_{b'}$ of $z$ at $b'$, a periodic point for $f_{b'}$. We deduce that $\pi_{b'}(z_{b'})\in \Per(f_{b})$. Since $\Per_{n}(f)_{b}=\Per_{n}(f_{b})$ for all $n\geq1$, we obtain the following: \textit{for any $b\in B$, the set $\{\pi_{b'}(z_{b'});\;z\in\Per(f),\; b'\textrm{ above }b\}$ is equal to $\Per(f_{b})$}.\\

Finally, for every $b\in B(\bar{\Q})$, we fix a complex embedding $\bar{\Q}(b)=\bar{\Q}\rightarrow\C$. This enables us to speak about the characteristic exponents and the Lyapunov exponent of $f_{b}$. Moreover, for any finite morphism $B'\rightarrow B$ and any $b'\in B'(\bar{\Q})$ above $b$, we may identify the $\bar{\Q}$-points of $\Ps^{1}_{b'}$ with those of $\Ps^{1}_{b}$, giving us an archimedean metric on $\Ps^{1}_{b'}(\bar{\Q})$ compatible with the one on $\Ps^{1}_{b}(\bar{\Q})$.

\begin{proposition}
    \label{prop4}
    Let $f$ be a non-exceptional endomorphism of degree $d\geq 2$, defined over $\bar{\Q}(B)$ as above.
    \begin{enumerate}
        \item Let $z\in\Per(f)$ be defined over a finite extension $\bar{\Q}(B')$ of $\bar{\Q}(B)$, and suppose that $\mathcal{O}:=C(f^{\infty})\cdot z$ is infinite. For any $b'\in B'(\bar{\Q})$ we have $\chi_{f_{b'}}(P')\leq L_{f_{b'}}$, where $P'$ is the $f_{b'}$-periodic cycle of the specialization of $z$ at $b'$.
        \item There exist infinitely many finite $C(f^{\infty})$-orbits in $\Per(f)$.
    \end{enumerate}
\end{proposition}

\begin{proof}[Proof of (i).]
    Let $z\in\Per(f)$ be defined over $\bar{\Q}(B')$ and set $\mathcal{O}:=C(f^{\infty})\cdot z$. Any $w\in \mathcal{O}$ is defined over a finite extension $\bar{\Q}(B'')$ of $\bar{\Q}(B')$. For any $b''\in B''$ that is above some $b'\in B'$, we may consider $\pi_{b''}(w_{b''})\in\Per(f_{b'})$. This allows us to define
    \[\mathcal{O}_{b'}:=\{\pi_{b''}(w_{b''});\;w\in\mathcal{O},\;b''\textrm{ above }b'\}\subset \Per(f_{b'})\]
    for any $b'\in B'$.
    
    \underline{Step 1}. We start by proving the following: \textit{for any $b'\in B'$ and any $n\geq1$,}
    \[\#(\mathcal{O}\cap\Per_{n}(f))\leq \#\mathcal{O}_{b'}+2d-2.\]
    
    Let $w\in\mathcal{O}$. It is of the form $w=g(z)$ for some $g\in C(f^{\infty})$. Now, $g$ belongs to the algebraic locus $\{h\in\mathrm{Rat}_{\deg(g)}\;|\; hf^{m}=f^{m}h\}$ for some integer $m\geq1$. This locus is finite as a consequence of \cite[Theorem 2]{Lev90}, \cite{Lev01}, so $g$ is defined over a finite extension of $\bar{\Q}(B')$. We can therefore make sense of the endomorphism $g^{\sigma}$, where $\sigma$ is an element of the absolute Galois group $G_{\bar{\Q}(B')}$. Given that $f$ is defined over $\bar{\Q}(B')$, we have $g^{\sigma}\in C(f^{\infty})$, and so $\sigma(w)=g^{\sigma}(z)\in\mathcal{O}$. From this discussion we deduce that $\mathcal{O}$ is $G_{\bar{\Q}(B')}$-invariant.
    
    For any $n\geq1$, let $\mathcal{P}^{(n)}$ be the Zariski closure of $\mathcal{O}\cap\Per_{n}(f)$ in $\Ps^{1}_{B'}$, and write $\mathcal{P}_{b'}^{(n)}:=\mathcal{P}^{(n)}\cap\Ps^{1}_{b'}$ for any $b'\in B'$. As before, $\mathcal{P}^{(n)}\rightarrow B'$ is proper and so
    \[b'\in B'\mapsto \sum_{z\in\mathcal{P}_{b'}^{(n)}}\mathrm{mult}_{z}(\mathcal{P}_{b'}^{(n)})\in \N\]
    is upper semi-continuous. Evaluating this map at the generic point $\eta'$ of $B'$ yields
    \begin{equation}
        \label{eq4}
        \sum_{z\in\mathcal{P}_{\eta'}^{(n)}}\mathrm{mult}_{z}(\mathcal{P}_{\eta'}^{(n)})\leq\sum_{z\in\mathcal{P}_{b'}^{(n)}}\mathrm{mult}_{z}(\mathcal{P}_{b'}^{(n)})
    \end{equation}
    for any $b'\in B'$. Note that, by the $G_{\bar{\Q}(B')}$-invariance of $\mathcal{O}\cap\Per_{n}(f)$, we have $\mathcal{P}_{\eta'}^{(n)}=\mathcal{O}\cap\Per_{n}(f)$ and similarly, $\mathcal{P}_{b'}^{(n)}=\mathcal{O}_{b'}\cap \Per_{n}(f_{b'})$.

    Fix $b'\in B'$ and $n\geq1$, and consider the multiplicities of points in $\mathcal{P}_{b'}^{(n)}$. By the Leau-Fatou theory of parabolic points, any $f_{b'}$-periodic point of multiplicity $m$ attracts $m-1$ critical points of $f_{b'}$ (see \cite[Theorem 10.15]{Mil06b}). There are at most $\#\mathcal{O}_{b'}$ underlying points in $\mathcal{P}_{b'}^{(n)}$, and every increment in multiplicity must be accounted for by an $f_{b'}$-critical point. Since $f_{b'}$ has degree $d$, it has at most $2d-2$ critical points, and thus
    \[\sum_{z\in\mathcal{P}_{b'}^{(n)}}\mathrm{mult}_{z}(\mathcal{P}_{b'}^{(n)})\leq \#\mathcal{O}_{b'}+2d-2.\]
    Applying Inequality \eqref{eq4} we obtain
    \[\#(\mathcal{O}\cap\Per_{n}(f))\leq\sum_{z\in\mathcal{P}_{\eta'}^{(n)}}\mathrm{mult}_{z}(\mathcal{P}_{\eta'}^{(n)})\leq\sum_{z\in\mathcal{P}_{b'}^{(n)}}\mathrm{mult}_{z}(\mathcal{P}_{b'}^{(n)})\leq \#\mathcal{O}_{b'}+2d-2.\]

    \underline{Step 2}. For the remainder of the proof, let us assume that $\mathcal{O}$ is infinite. If 
    for any $b'\in B'$, $\mathcal{O}_{b'}$ were finite, then by the previous step, we would have
    \[\lim_{k\rightarrow\infty}\#(\mathcal{O}\cap\Per_{k!}(f))\leq \#\mathcal{O}_{b'}+2d-2<+\infty\]
    which is absurd.
    
    We have shown that $\mathcal{O}_{b'}$ is infinite for any $b'$. Now, let $b'\in B'(\bar{\Q})$. Denote by $p$ the period of $z$, and by $\lambda$ its multiplier. Let $z_{b'}$ be the specialization of $z$ at $b'$, $P'$ its corresponding cycle and $\lambda_{b'}$ its multiplier. In order to prove assertion $(i)$ we may clearly assume that $\lambda_{b'}$ is not $0$ nor a root of unity. Consider $w\in \mathcal{O}$ with period $p_{w}$ and multiplier $\lambda_{w}$, defined over $\bar{\Q}(B'')$. From the proof of Proposition \ref{prop3} $(i)$ we know that there exists some $n$ such that 
    \[\lambda_{w}^{np/p_{w}}=\lambda^{n}.\]
    This equality remains true after specialization at $b'$: for any $w_{b'}\in\Per(f_{b'})$ of the form $\pi_{b''}(w_{b''})$, with cycle $P_{w}$ and multiplier $\lambda_{w,b'}$, we have
    \[\lambda_{w,b'}^{np/p_{w}}=\lambda_{b'}^{n}\]
    and thus
    \[\chi_{f_{b'}}(P_{w})=\chi_{f_{b'}}(P').\]
    Here we use the fact that the period of $w_{b'}$ is $p_{w}$, the same as the period of its generic counterpart $w$. Indeed, if the period were to drop, $w_{b'}$ would be a parabolic point and thus $\lambda_{b'}$ would be a root of unity, a possibility we have excluded.
    
    \underline{Conclusion}. We have the equality $\chi_{f_{b'}}(P_{w})=\chi_{f_{b'}}(P')$ for infinitely many $f_{b'}$-periodic cycles $P_{w}$. We may now apply the equidistribution argument in the proof of Proposition \ref{prop3} $(i)$, working over a number field over which $f_{b'}$ and $\lambda_{b'}$ are defined, in order to obtain $\chi_{f_{b'}}(P')\leq L_{f_{b'}}$. This establishes assertion $(i)$. 
\end{proof}

\begin{proof}[Proof of (ii).]
    Let $b\in B(\bar{\Q})$, and consider a point $z_{b}\in\Per(f_{b})$. There exists some $z\in \Per(f)$ defined over an extension $\bar{\Q}(B')$ of $\bar{\Q}(B)$, and a point $b'$ above $b$, such that $\pi_{b'}(z_{b'})=z_{b}$. If $C(f^{\infty})\cdot z$ is infinite, then by assertion $(i)$, $\chi_{f_{b'}}(P')\leq L_{f_{b'}}$, where $P'$ is the $f_{b'}$-periodic cycle containing $z_{b'}$, and thus $\chi_{f_{b}}(P)\leq L_{f_{b}}$, where $P$ is the $f_{b}$-periodic cycle containing $z_{b}$.
    
    If all but finitely many $z\in \Per(f)$ had an infinite $C(f^{\infty})$-orbit, then we would have $\chi_{f_{b}}(P)\leq L_{f_{b}}$ for all but finitely many periodic cycles of $f_{b}$, and hence $f_{b}$ would be contained in the exceptional locus $\mathcal{E}$ as a consequence of Theorem \ref{theo3}. This would hold true for any $b\in B(\bar{\Q})$: by Lemma \ref{lem3}, $\mathcal{E}$ is Zariski closed and so the generic point of the family $(f_{b})_{b\in B}$, which is $f$, would also be contained in $\mathcal{E}$. This contradicts our initial assumption that $f$ is non-exceptional, concluding the proof of $(ii)$.
\end{proof}

\begin{remark}
    \label{rem1}
    Proposition \ref{prop2} becomes trivial if we replace $C(f^{\infty})$ with $C(f)$. Indeed, if $z$ is a periodic point for $f$ of period $p$, then the orbit $C(f)\cdot z$ consists of periodic points for $f$ of period dividing $p$. Thus, $C(f)\cdot z$ is finite, of size at most $\deg(f)^{p}+1$.
\end{remark}
\newpage

\section{A criterion for equality for endomorphisms}

The purpose of this section is to prove the following strengthening of Lemma \ref{lem1}: if some non-linear, non-exceptional endomorphisms $f,g$ have the same degree, the same measure of maximal entropy, and three common fixed points in their common Julia set, then $f=g$ (see Proposition \ref{prop5} below). Our proof is based on a result of Ritt \cite{Rit23}, for which we also provide a new proof.

\subsection{Ritt sequence} Let $f,g\in\End(\Ps^{1})^{+}$ be commuting endomorphisms. The map
\[(f,g):\Ps^{1}\rightarrow\Ps^{1}\!\times\Ps^{1}:\;z\mapsto(f(z),g(z))\]
has as an image a rational, possibly singular curve $\Gamma$. Denote by $(a,b):\Ps^{1}\rightarrow \Gamma$ its normalization map. Then $(f,g)$ factors through $(a,b)$: there exists $u:\Ps^{1}\rightarrow\Ps^{1}$ such that $f=au$ and $g=bu$. Since $fg=gf$, we have $aubu=buau$ and by right cancellation, $aub=bua$.

Define $f_{1}=ua$ and $g_{1}=ub$. Note that $\deg(f_{1})=\deg(f)$ and $\deg(g_{1})=\deg(g)$. Since $aub=bua$, composing with $u$ on the left yields $f_{1}g_{1}=g_{1}f_{1}$. More generally, define by induction
\[(f_{0},g_{0})=(f,g),\quad (f_{n+1},g_{n+1})=(u_{n}a_{n},u_{n}b_{n})\quad\mathrm{where}\quad(f_{n},g_{n})=(a_{n}u_{n},b_{n}u_{n})\]
and $z\mapsto(a_{n}(z),b_{n}(z))$ is a generically injective parametrization of the curve $\Gamma_{n}:=\{(f_{n}(z),g_{n}(z));z\in\Ps^{1}\}$. We call $(a_{n},b_{n},u_{n})_{n\geq0}$ the \textit{Ritt sequence} attached to $(f,g)$. It is uniquely defined, up to the changes $(a_{n},b_{n},u_{n})\mapsto (a_{n}\sigma,b_{n}\sigma,\sigma^{-1}u_{n})$ for any linear $\sigma$. As before, we have
\begin{equation}
    \label{eq5}
    \forall n\geq0,\quad a_{n}u_{n}b_{n}=b_{n}u_{n}a_{n}
\end{equation}
so each $(f_{n},g_{n})$ is a commuting pair with $\deg(f_{n})=\deg(f)$ and $\deg(g_{n})=\deg(g)$. Since $a_{n}u_{n}b_{n}=a_{n}b_{n+1}u_{n+1}$ and $b_{n}u_{n}a_{n}=b_{n}a_{n+1}u_{n+1}$, we also obtain the relation
\begin{equation}
    \label{eq6}
    \forall n\geq0,\quad a_{n}b_{n+1}=b_{n}a_{n+1}.
\end{equation}
Equation \eqref{eq6} tells us that $\Gamma_{n+1}:=\{(a_{n+1}(z),b_{n+1}(z));z\in\Ps^{1}\}$ is contained in the set $\Delta_{n}:=\{(x,y)\in\Ps^{1}\!\times\Ps^{1}\;|\;b_{n}(x)=a_{n}(y)\}$. Recall that the \textit{bidegree} of a (reducible) curve $C\subset\Ps^{1}\times\Ps^{1}$ is the pair $(d_{1},d_{2})$, where $d_{1}$ is the intersection number of $C$ with a vertical fiber $\{*\}\times\Ps^{1}$, and similarly $d_{2}$ is the intersection number of $C$ with a horizontal fiber $\Ps^{1}\times\{*\}$. For any $n\geq0$, the bidegrees of $\Gamma_{n}$ and $\Delta_{n}$ are both equal to $(\deg(a_{n}),\deg(b_{n}))$. The inclusion $\Gamma_{n+1}\subset\Delta_{n}$ therefore implies that
\[\deg(a_{n+1})\leq\deg(a_{n}),\quad\deg(b_{n+1})\leq\deg(b_{n})\]
In particular, the sequences $(\deg(a_{n}))_{n\geq0}$ and $(\deg(b_{n}))_{n\geq0}$ are eventually constant.\\

Ritt introduced this sequence in order to put commuting pairs into a more manageable form. For simplicity, we will state his result in the equal-degree case, $\deg(f)=\deg(g)$. This implies that $\deg(a_{n}u_{n})=\deg(b_{n}u_{n})$ for any $n\geq0$, and hence $\deg(a_{n})=\deg(b_{n})$ for any $n\geq0$.

\begin{theorem}[{\cite[§X]{Rit23}}]
    \label{theo5}
    Let $f,g\in\End(\Ps^{1})^{+}$ be a pair of commuting endomorphisms of equal degree. Denote by $(a_{n},b_{n},u_{n})_{n\geq0}$ its Ritt sequence, and let $r_{n}$ be the common degree of $a_{n}$ and $b_{n}$. If $f,g$ are non-exceptional, then $r_{n}=1$ for large enough $n$.
\end{theorem}

Ritt used this result in order to prove the theorem stated in the introduction, which we restate in the equal-degree case:

\begin{theorem}
    \label{theo6}
    Let $f,g\in\End(\Ps^{1})^{+}$ be a pair of commuting endomorphisms of equal degree. If $f,g$ are non-linear and non-exceptional, there exists an integer $p\geq1$ such that $f^{p}=g^{p}$.
\end{theorem}

Let us briefly mention how Theorem \ref{theo5} implies \ref{theo6}. Let $f,g$ be non-linear, non-exceptional, commuting endomorphisms of equal degree. Denote by $(a_{n},b_{n},u_{n})_{n\geq0}$ their Ritt sequence. By Theorem \ref{theo5} there exists some $n\geq0$ such that $\deg(a_{n})=\deg(b_{n})=1$. We can define $\sigma=a_{n}b_{n}^{-1}$, so that $f_{n}=\sigma g_{n}$. Now, $\sigma g_{n}g_{n}=g_{n}\sigma g_{n}$ so by right cancellation $\sigma g_{n}=g_{n}\sigma$. Since $\deg(g_{n})=\deg(g)\geq2$, $\sigma$ must be of finite order, say $p$, which yields
\[f_{n}^{p}=(\sigma g_{n})^{p}=\sigma^{p}g_{n}^{p}=g_{n}^{p}.\]
Since $f_{n}=u_{n-1}a_{n-1}$ and $g_{n}=u_{n-1}b_{n-1}$, we deduce that
\[f_{n-1}^{p}a_{n-1}=a_{n-1}(u_{n-1}a_{n-1})^{p}=a_{n-1}(u_{n-1}b_{n-1})^{p}=(b_{n-1}u_{n-1})^{p}a_{n-1}=g_{n-1}^{p}a_{n-1}\]
where the next-to-last equality comes from Equation \eqref{eq5}. We obtain $f_{n-1}^{p}=g_{n-1}^{p}$, and by a backwards induction, $f^{p}=g^{p}$.\\

Theorem \ref{theo5} is proved by Ritt by means of a case-by-case analysis which doesn't lend itself to generalization. On the other hand, Theorem \ref{theo6} has been established by different means by Eremenko \cite{Ere90}, following the work of Fatou and Julia \cite{Fat23}, \cite{Jul22}. This new approach has also yielded results for commuting pairs of endomorphisms in higher dimensions \cite{DS02}, \cite{Kau18}. We therefore find it relevant to provide a proof that Theorem \ref{theo6} implies Theorem \ref{theo5}. We will do so using the language of \textit{finitary correspondences}.

\subsection{Correspondences} In our setting, a \textit{correspondence} $c$ is given by a pair of dominant maps $a,b:\Ps^{1}\rightarrow\Ps^{1}$:
\[\begin{tikzcd}
	& {\Ps^{1}} \\
	{\Ps^{1}} && {\Ps^{1}}
	\arrow["a"', from=1-2, to=2-1]
	\arrow["b", from=1-2, to=2-3]
	\arrow["c", dashed, from=2-1, to=2-3]
\end{tikzcd}\]
We will also denote this correspondence by $c=ba^{-1}$. We can interpret $c$ as a multi-valued function: for any $z\in \Ps^{1}$ we let $c(z):=b(a^{-1}(z))$. Equivalently, $w\in c(z)$ if and only if there exists $x\in\Ps^{1}$ such that $a(x)=z$, $b(x)=w$. The \textit{graph} of $c$ is the curve in $\Ps^{1}\times\Ps^{1}$ defined by
\[\Gamma_{c}:=\{(z,w)\in\Ps^{1}\times\Ps^{1}\;|\;w\in c(z)\}=\{(a(x),b(x));\;x\in\Ps^{1}\}.\]
More generally, we define by induction $c^{0}(z)=\{z\}$, $c^{k+1}(z)=\bigcup_{w\in c^{k}(z)}c(w)$.
Equivalently, $w\in c^{k}(z)$ if and only if there exists a sequence $x_{1},\cdots,x_{k}\in\Ps^{1}$ such that $a(x_{1})=z$, $b(x_{i})=a(x_{i+1})$ for all $1\leq i<k$, and $b(x_{k})=w$. We define the graph $\Gamma_{c}^{k}$ as $\{(z,w)\in\Ps^{1}\times\Ps^{1}\;|\;w\in c^{k}(z)\}$. It is also a (reducible) curve. Indeed, letting $M$ be the subset of points in $(\Ps^{1})^{4}$ whose second and third coordinates coincide, and letting $\pi:(\Ps^{1})^{4}\rightarrow\Ps^{1}\times\Ps^{1}$ be the projection onto the first and fourth coordinates, we have
\[\Gamma_{c}^{k+1}=\pi\left(\left(\Gamma_{c}^{k}\times\Gamma_{c}\right)\cap M\right).\]
Finally, the \textit{forward orbit} of a point $z\in\Ps^{1}$ with respect to a correspondence $c$ is defined as
\[\mathrm{O}_{c}(z):=\bigcup_{k\geq0}c^{k}(z).\]
It is the smallest set containing $z$ and invariant under $c$.

\begin{remark}
    A more thorough analysis of correspondences would require us to define the graphs $\Gamma_{c}^{k}$ as effective divisors, in order to account for multiplicity in the images of $c^{k}$. We will limit ourselves to the description without multiplicity given above.
\end{remark}

Following \cite{Bel23}, we call a correspondence $c$ \textit{finitary} if $\mathrm{O}_{c}(z)$ is finite for every $z\in\Ps^{1}$. In particular, this implies that the sequence
\[\bigcup_{i=0}^{k}\Gamma_{c}^{i}\;;\quad k\geq0\]
is eventually constant, equal to a closed curve $\Gamma_{c}^{\infty}$. Letting $(s_{1},s_{2})$ denote the bidegree of $\Gamma_{c}^{\infty}$, we define the \textit{generic orbit size} of $c$ as $s_{c}:=s_{1}$. As the name implies, for a general choice of $z$, $\mathrm{O}_{c}(z)$ is of size $s_{c}$.

\begin{remark}
    J. Bellaïche defines finitary correspondences in \cite{Bel23} using a slightly different language. Let $c^{-1}:=ab^{-1}$, and define the \textit{grand orbit} $\mathrm{GO}_{c}(z)$ of $z$ with respect to $c$ as the smallest set containing $z$ and invariant under both $c$ and $c^{-1}$. Bellaïche calls a correspondence finitary if $\mathrm{GO}_{c}(z)$ is finite for every $z$.

    These two definitions are equivalent. Indeed, if $\mathrm{GO}_{c}(z)$ is finite for every $z$ then so is $\mathrm{O}_{c}(z)$. Conversely, suppose that the forward orbits of $c$ are finite. Observe that the graphs $\Gamma_{c^{-1}}^{k}$ are the reflexions of the graphs $\Gamma_{c}^{k}$ along the main diagonal of $\Ps^{1}\times\Ps^{1}$, and therefore the forward orbits of $c^{-1}$ are also finite. After replacing $c$ with $c^{-1}$ if necessary, we may assume that the bidegree $(r_{1},r_{2})$ of $\Gamma_{c}$ satisfies $r_{1}\geq r_{2}.$ A forward orbit $\mathrm{O}_{c}(z)$ can be seen as a directed graph where $w_{1}\rightarrow w_{2}$ whenever $w_{2}\in c(w_{1})$. For a general choice of $z$, every point of $\mathrm{O}_{c}(z)$ has out-degree $r_{1}$ and in-degree at most $r_{2}\leq r_{1}$. Since the number of edges is $r_{1}\cdot\#\mathrm{O}_{c}(z)$, the in-degrees must be all equal to $r_{1}$. In particular $r_{1}=r_{2}$, and more importantly, $\mathrm{O}_{c}(z)$ is invariant by $c^{-1}$ so $\mathrm{O}_{c}(z)=\mathrm{GO}_{c}(z)$.
\end{remark}

We return to the proof that Theorem \ref{theo6} implies Theorem \ref{theo5}. Let $f,g$ be commuting, non-exceptional endomorphisms of equal degree $d\geq2$. By Theorem \ref{theo6}, there exists some integer $p\geq1$ such that $f^{p}=g^{p}$. We denote by $(a,b,u)$ the first term in the Ritt sequence associated with $(f,g)$.

\begin{lemma}
    \label{lem4}
    The correspondence $c=ba^{-1}$ is finitary, with generic orbit size $s_{c}$ bounded from above by $pd^{p}$.
\end{lemma}

\begin{proof}
    Let $z$ be a general point in $\Ps^{1}$, and let us show that $s_{c}=\#\mathrm{O}_{c}(z)\leq pd^{p}$. Let $w\in c^{p}(z)$. By definition, there exist $x_{1},\cdots, x_{p}\in \Ps^{1}$ such that $a(x_{1})=z$, $b(x_{i})=a(x_{i+1})$ for all $1\leq i<p$, and $b(x_{p})=w$. For every $1\leq i\leq p$, choose a point $x_{i}'\in u^{-1}(x_{i})$. Since $(f,g)=(au,bu)$, we have $f(x_{1}')=z$, $g(x_{i}')=f(x_{i+1}')$ for all $1\leq i<p$, and $g(x_{p}')=w$. Now,
    \begin{align*}
        f^{p}(z)=g^{p}(z)=g^{p}(f(x_{1}')) &=fg^{p-1}(g(x_{1}'))\\
        &=fg^{p-1}(f(x_{2}'))\\
        &=f^{2}g^{p-2}(g(x_{2}'))\\
        &\cdots\\
        &=f^{p}(g(x_{p}'))=f^{p}(w)
    \end{align*}
    so $f^{p}(z)=f^{p}(w)$. By the same argument, if $w_{1}\in c^{k_{1}}(z)$, $w_{2}\in c^{k_{2}}(z)$ are such that $k_{1}\equiv k_{2}\mod p$, then $f^{p}(w_{1})=f^{p}(w_{2})$. Choosing $w_{i}\in c^{i}(z)$ for every $1\leq i\leq p$, we obtain
    \[\mathrm{O}_{c}(z)\subset\bigcup_{i=1}^{p}f^{-p}(f^{p}(w_{i}))\]
    and thus $s_{c}=\#\mathrm{O}_{c}(z)\leq pd^{p}$.
\end{proof}

Let $(a_{n},b_{n},u_{n})_{n\geq0}$ be the full Ritt sequence associated with $(f,g)$. Recall that for any $n\geq0$, $(f_{n},g_{n})=(a_{n}u_{n},b_{n}u_{n})$ is a commuting pair of equal degree $d$. Moreover, the equality $f^{p}=g^{p}$ can be written $(a_{0}u_{0})^{p}=(b_{0}u_{0})^{p}$: cancelling $u_{0}$ on the right and composing by $u_{0}$ on the left, we obtain $f_{1}^{p}=g_{1}^{p}$, and by induction, $f_{n}^{p}=g_{n}^{p}$. Therefore, Lemma \ref{lem4} applies to every pair $(f_{n},g_{n})$: the correspondence $c_{n}=b_{n}a_{n}^{-1}$ is finitary with generic orbit size $s_{c_{n}}$ bounded from above by $pd^{p}$.

\begin{lemma}
    \label{lem5}
    Let $n\geq0$ be an index in the Ritt sequence $(a_{n},b_{n},u_{n})_{n\geq0}$ such that $\deg(a_{n})=\deg(a_{n+1})=:r$. Then the generic orbit sizes of $c_{n}$ and $c_{n+1}$ satisfy $s_{c_{n+1}}\geq rs_{c_{n}}$.
\end{lemma}

\begin{proof}
    Recall that the graph $\Gamma_{n+1}=\Gamma_{c_{n+1}}=\{(a_{n+1}(z),b_{n+1}(z));\;z\in\Ps^{1}\}$ is contained in $\Delta_{n}=\{(x,y)\in\Ps^{1}\times\Ps^{1}\;|\;b_{n}(x)=a_{n}(y)\}$. Since we are assuming that $\deg(a_{n})=\deg(a_{n+1})=r$ (and thus $\deg(b_{n})=\deg(b_{n+1})=r$) the bidegrees of $\Gamma_{n+1}$ and $\Delta_{n}$ are both equal to $(r,r)$ and hence $\Gamma_{n+1}=\Delta_{n}$. This translates to the following property: for every $(x,y)$ such that $b_{n}(x)=a_{n}(y)$, there exists a $z\in\Ps^{1}$ such that $a_{n+1}(z)=x$ and $b_{n+1}(z)=y$:
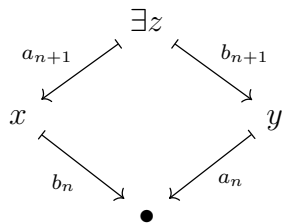
    \[\begin{tikzcd}
	& {\exists z} \\
	x && y \\
	& \bullet
	\arrow["{a_{n+1}}"', maps to, from=1-2, to=2-1]
	\arrow["{b_{n+1}}", maps to, from=1-2, to=2-3]
	\arrow["{b_{n}}"', maps to, from=2-1, to=3-2]
	\arrow["{a_{n}}", maps to, from=2-3, to=3-2]
    \end{tikzcd}\]
    Let $z$ be a general point in $\Ps^{1}$, and fix some $z^{'}\in a_{n}^{-1}(z)$. Since $z$ is general, the points in $\mathrm{O}_{c_{n}}(z)$ are not critical values of $a_{n}$, so $\#a_{n}^{-1}(\mathrm{O}_{c_{n}}(z))=rs_{c_{n}}$. To prove the lemma, it is therefore sufficient to show that $a_{n}^{-1}(\mathrm{O}_{c_{n}}(z))\subset \mathrm{O}_{c_{n+1}}(z')$.

    Let $\tilde{z}=b_{n}(z')$. The orbit $\mathrm{O}_{c_{n}}(\tilde{z})$ is contained in $\mathrm{O}_{c_{n}}(z)$, but it is also of size $s_{c_{n}}$ so $\mathrm{O}_{c_{n}}(\tilde{z})=\mathrm{O}_{c_{n}}(z)$. Let $w'\in a_{n}^{-1}(\mathrm{O}_{c_{n}}(z))$, and write $w=a_{n}(w')$. We have $w\in \mathrm{O}_{c_{n}}(\tilde{z})$, so there exists a sequence $x_{1},\cdots,x_{k}\in\Ps^{1}$ such that $a_{n}(x_{1})=\tilde{z}$, $b_{n}(x_{i})=a_{n}(x_{i+1})$ for all $1\leq i<k$, and $b_{n}(x_{k})=w$. In other words, the pairs $(z',x_{1})$, $(x_{i},x_{i+1})$ for every $1\leq i<k$, and $(x_{k},w')$ are all contained in $\ \Delta_{n}=\Gamma_{n+1}$. We can therefore find $x_{0}',\cdots,x_{k+1}'$ such that $(a_{n+1}(x_{i}'),b_{n+1}(x_{i}'))=(z',x_{0})$, $(x_{i},x_{i+1})$, or $(x_{k},w')$ depending on whether $i=0$, $1\leq i<k$, or $i=k$:
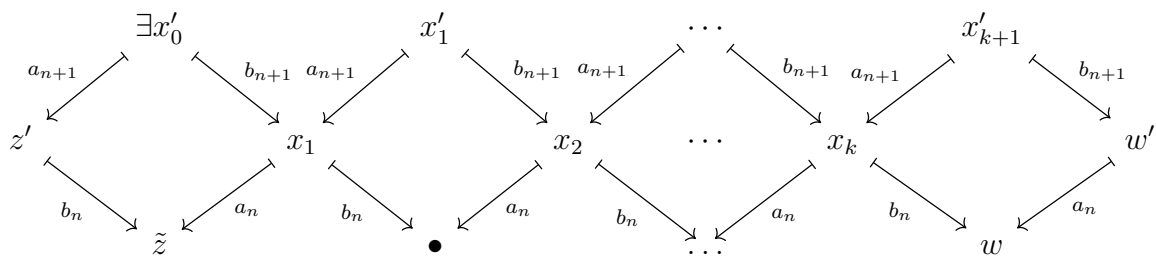
    \[\begin{tikzcd}
	& {\exists x_{0}'} && {x_{1}'} && \cdots && {x_{k+1}'} \\
	{z'} && {x_{1}} && {x_{2}} & \cdots & {x_{k}} && {w'} \\
	& {\tilde{z}} && \bullet && \dots && w
	\arrow["{a_{n+1}}"', maps to, from=1-2, to=2-1]
	\arrow["{b_{n+1}}", maps to, from=1-2, to=2-3]
	\arrow["{a_{n+1}}"', maps to, from=1-4, to=2-3]
	\arrow["{b_{n+1}}", maps to, from=1-4, to=2-5]
	\arrow["{a_{n+1}}"', maps to, from=1-6, to=2-5]
	\arrow["{b_{n+1}}", maps to, from=1-6, to=2-7]
	\arrow["{a_{n+1}}"', maps to, from=1-8, to=2-7]
	\arrow["{b_{n+1}}", maps to, from=1-8, to=2-9]
	\arrow["{b_{n}}"', maps to, from=2-1, to=3-2]
	\arrow["{a_{n}}", maps to, from=2-3, to=3-2]
	\arrow["{b_{n}}"', maps to, from=2-3, to=3-4]
	\arrow["{a_{n}}", maps to, from=2-5, to=3-4]
	\arrow["{b_{n}}"', maps to, from=2-5, to=3-6]
	\arrow["{a_{n}}", maps to, from=2-7, to=3-6]
	\arrow["{b_{n}}"', maps to, from=2-7, to=3-8]
	\arrow["{a_{n}}", maps to, from=2-9, to=3-8]
    \end{tikzcd}\]
    This shows that $w'\in \mathrm{O}_{c_{n+1}}(z')$ as desired.
\end{proof}

We can now conclude our proof that Theorem \ref{theo6} implies \ref{theo5}. Indeed, there exists an index $n_{0}$ in Ritt's sequence such that for every $k\geq 0$, $\deg(a_{n_{0}+k})=\deg(a_{n_{0}}):=r$. By a repeated application of Lemma \ref{lem5}, we have $s_{c_{n_{0}+k}}\geq r^{k}s_{c_{n_{0}}}$. At the same time, we know by Lemma \ref{lem4} that $s_{c_{n_{0}+k}}\leq pd^{p}$, so this can only occur if $r=1$.

\begin{proposition}
    \label{prop5}
    Let $f,g\in\End(\Ps^{1})^{+}$ be non-linear, non-exceptional endomorphisms. Assume that $f$ and $g$ have the same degree, the same measure of maximal entropy, and three common fixed points $z^{(1)},z^{(2)},z^{(3)}$ in their common Julia set. Then $f=g$.
\end{proposition}

\begin{proof}
    By Lemma \ref{lem1}, $f$ and $g$ commute. Let $(a_{n},b_{n},u_{n})_{n\geq0}$ be the associated Ritt sequence, and write $(f_{n},g_{n})=(a_{n}u_{n},b_{n}u_{n})$ as usual. For any $i\in\{1,2,3\}$, we have
    \[f_{1}(u_{0}(z^{(i)}))=u_{0}a_{0}u_{0}(z^{(i)})=u_{0}(f(z^{(i)}))=u_{0}(z^{(i)})\]
    so $u_{0}(z^{(i)})$ is a fixed point of $f_{1}$, and the same argument shows that it is a fixed point of $g_{1}$ as well. Moreover, $u_{0}(z^{(i)})\neq u_{0}(z^{(j)})$ whenever $i\neq j$, for otherwise
    \[z^{(i)}=a_{0}(u_{0}(z^{(i)}))=a_{0}(u_{0}(z^{(j)}))=z^{(j)}.\]
    More generally, define
    \[z_{0}^{(i)}=z^{(i)},\quad z_{n+1}^{(i)}=u_{n}(z_{n}^{(i)}).\]
    Then $z_{n}^{(1)},z_{n}^{(2)},z_{n}^{(3)}$ are distinct fixed points of $f_{n}$ and $g_{n}$. By Theorem \ref{theo5}, there exists some $n\geq0$ such that $\deg(a_{n})=\deg(b_{n})=1$. Let $\sigma=a_{n}b_{n}^{-1}$, so that $f_{n}=\sigma g_{n}$. It follows that $z_{n}^{(1)},z_{n}^{(2)},z_{n}^{(3)}$ are fixed points of $\sigma$ and thus $\sigma$ must be equal to the identity: we have $f_{n}=g_{n}$. Applying Equation \eqref{eq5},
    \[f_{n-1}a_{n-1}=a_{n-1}u_{n-1}a_{n-1}=a_{n-1}u_{n-1}b_{n-1}=b_{n-1}u_{n-1}a_{n-1}=g_{n-1}a_{n-1}\]
    so $f_{n-1}=g_{n-1}$, and by a backwards induction, $f=g$. 
\end{proof}

\begin{remark}
    Proposition \ref{prop5} also holds when $f$ and $g$ are simultaneously conjugate to power maps or Chebyshev polynomials. Indeed, in both cases the equality $\deg(f)=\deg(g)$ directly implies that $f=\sigma g$ for some linear $\sigma$, and thus the existence of three common fixed points yields $f=g$.

    If $f$ and $g$ are both Lattès maps contained in the same semigroup $L(E,\Gamma)$, Proposition \ref{prop5} holds after a slight modification. Let $R$ be the ring of endomorphisms of $E$, and denote by $R^{\times}$ its group of units. We define the \textit{Lattès degree} as $L(E,\Gamma)\rightarrow (R\backslash\{0\})/R^{\times}$, $\ell_{\alpha,t}\mapsto \alpha\mod R^{\times}$. If $R=\Z$ this is simply the square of the degree map. Now, if $f$ and $g$ are of equal Lattès degree then $f=\sigma g$ for some linear $\sigma$, and if they have three common fixed points, $f=g$.
\end{remark}

\section{Proof of Theorems}

\subsection{Virtual cyclicity of \texorpdfstring{$C(f^{\infty})$}{C(f∞)}} Our Theorem \ref{theo1} is a consequence of Propositions \ref{prop2} and \ref{prop5}.

\begin{proof}[Proof of Theorem \ref{theo1}] Let $f$ be a non-linear, non-exceptional endomorphism of $\Ps^{1}$. Our goal is to show that $C(f^{\infty})$ is virtually cyclic. Let $\deg(f)=p_{1}^{\alpha_{1}}\cdots p_{s}^{\alpha_{s}}$ be the prime factorization of the degree of $f$, and define $d_{0}:=p_{1}^{\beta_{1}}\cdots p_{s}^{\beta_{s}}$, where $(\beta_{1},\cdots,\beta_{s})$ is obtained by dividing $(\alpha_{1},\cdots,\alpha_{s})$ by their largest common divisor. By Proposition \ref{prop1}, $(i)\Rightarrow(ii)$, for any non-linear $g\in C(f^{\infty})$, there exists $m,n\geq1$ such that $g^{m}=f^{n}$. Thus, there exists some $k\geq1$ such that $\deg(g)=d_{0}^{k}$. This allows us to define the \textit{logarithmic degree}
\[\ell: C(f^{\infty})\rightarrow\N,\; g\mapsto\log_{d_{0}}(\deg(g)).\]
By Proposition \ref{prop2}, there are infinitely many $z\in\Per(f)$ such that $C(f^{\infty})\cdot z$ is finite. We may therefore choose a finite orbit $\mathcal{O}:=C(f^{\infty})\cdot z$ of size at least $3$ contained in the Julia set $J$ of $f$. 
For any non-linear $g\in C(f^{\infty})$, we have $\mathcal{O}\subset \Per(g)$, as a consequence of Proposition \ref{prop1}, $(i)\Rightarrow(iii)$. This shows that $C(f^{\infty})$ acts on $\mathcal{O}$ by permutations.

Let $\mathfrak{S}(\mathcal{O})$ be the group of permutations of $\mathcal{O}$. Let $F=f^{N}$ be such that $F_{|\mathcal{O}}=\mathrm{id}_{|\mathcal{O}}$, and consider the map
\[\varphi:C(f^{\infty})\rightarrow\{0,1,\cdots,\ell(F)-1\}\times\mathfrak{S}(\mathcal{O}),\quad g\mapsto (\ell(g)\;\mathrm{mod}\;\ell(F),\; g_{|\mathcal{O}}).\]
We may choose a finite set of representatives $\{h_{1},\cdots,h_{r}\}$ for each value taken by $\varphi$, each one of minimal degree in its fiber. Let us prove that $C(f^{\infty})=\{h_{1},\cdots,h_{r}\}\langle f\rangle$.

We begin by stating two general facts (see Subsections \ref{ss21} and \ref{ss23}). Any non-linear $g\in C(f^{\infty})$ has the same set of preperiodic points as $f$, so in view of Lemma \ref{lem2}, every non-linear element in $C(f^{\infty})$ is non-exceptional. Moreover, recall that every non-linear element of $C(f^{\infty})$ has the same measure of maximal entropy as $f$.

Let $g\in C(f^{\infty})$. There exists some $1\leq i\leq r$ such that $\varphi(g)=\varphi(h_{i})$. In particular, there is some $k\geq0$ such that $\ell(g)=\ell(h_{i})+k\ell(F)$. Set $h:=h_{i}F^{k}$, and let us prove that $g=h$. By construction, $\deg(g)=\deg(h)$. Since $F_{|\mathcal{O}}=\mathrm{id}_{|\mathcal{O}}$, we have $g_{|\mathcal{O}}=h_{|\mathcal{O}}$. If we denote by $p\geq1$ the order of $g_{|\mathcal{O}}$, then $g^{p}$ and $hg^{p-1}$ are of equal degree and act as the identity on $\mathcal{O}$. In particular, they have three common fixed points. If $g,h$ are both linear, it immediately follows that $g^{p}=hg^{p-1}=\mathrm{id}$ and hence $g=h$. Suppose then that $g,h$ are non-linear. Then $g^{p}$ and $hg^{p-1}$ are both non-exceptional, have a common measure of maximal entropy, and since $\mathcal{O}\subset J$, they have three common fixed points in their common Julia set. Hence, we may apply Proposition \ref{prop5}: we have $g^{p}=hg^{p-1}$, and by right cancellation, $g=h$. This concludes the proof.
\end{proof}

\begin{remark}
    \label{rem2}
    The arguments used throughout this proof still hold if we replace $C(f^{\infty})$ with any subsemigroup of $C(f^{\infty})$ containing $f$. In particular, we deduce that $C(f)$ is of the form $\{g_{1},\cdots,g_{s}\}\langle f\rangle$, as F. Pakovich established in \cite{Pak20}. Like Pakovich, we can show that the value of $s$ is bounded in terms of $\deg(f)$. Indeed, inspecting the proof of Theorem \ref{theo1}, we see that $s$ is controlled by $\deg(f)$ as well as the size of the chosen finite orbit $\mathcal{O}$. In turn, $\#\mathcal{O}$ can be controlled by the degree of $f$ when dealing with the semigroup $C(f)$, as a consequence Remark \ref{rem1}. We obtain from this an effective version of Ritt's theorem: there exists a computable function $\eta:\N\rightarrow\N$ such that for any non-linear, non-exceptional, commuting pair of endomorphisms $u,v$, we have $u^{m}=v^{n}$ where $n\leq\eta(\deg(u))$ and $m\leq\eta(\deg(v))$ (see \cite{Pak20} for details). On the other hand, we don't obtain a bound on the degrees of the endomorphisms $g_{1},\cdots,g_{s}$, as Pakovich does.
\end{remark}

\begin{proof}[Proof of Corollary \ref{cor1}]
    Let $f$ be an endomorphism which is not conjugate to a power map. Our goal is to show that there exists $N\geq1$ such that $C(f^{\infty})=C(f^{N})$. We divide the proof into four cases:

    \underline{Linear case}. Suppose that $f$ is linear. If $f$ is of finite order $N$ then $C(f^{\infty})=C(f^{N})=\End(\Ps^{1})^{+}$. We may therefore assume that $f$ is of infinite order. In particular it cannot commute with any non-linear endomorphism, so that every element in $C(f^{\infty})$ is linear. If $f$ has a single fixed point then up to conjugation it is of the form $z\mapsto z+\alpha$, and its centralizer is comprised of elements of the form $z\mapsto z+\beta$. Similarly, if $f$ has two fixed points then up to conjugation it is of the form $z\mapsto\alpha z$, and its centralizer is comprised of elements of the form $z\mapsto\beta z$. In both cases, we see that the description of the centralizer remains unchanged if we replace $f$ with a positive iterate, so $C(f^{\infty})=C(f)$.

    \underline{Non-exceptional case}. Suppose that $f$ is non-linear, non-exceptional. By Theorem \ref{theo1} there exists $h_{1},\cdots,h_{r}$ such that $C(f^{\infty})=\{h_{1},\cdots,h_{r}\}\langle f\rangle$. For every $1\leq i\leq r$ there is some $n_{i}$ such that $h_{i}\in C(f^{n_{i}})$. If we set $N=n_{1}\cdots n_{r}$, then $f^{N}$ commutes with every $h_{i}$. Since these elements generate $C(f^{\infty})$ alongside $f$, we obtain $C(f^{\infty})=C(f^{N})$.

    \underline{Chebyshev case}. Suppose that $f=\pm T_{d}$, $d\geq2$. Then by Lemma \ref{lem2} $(ii)$, every element in $C(f^{\infty})$ is of the form $\pm T_{e}$, $e\geq1$. Suppose that $d$ is even. Then $f$ is an even map: up to conjugation by $z\mapsto-z$, we have $f=T_{d}$ and $C(f)=\{T_{e};\;e\geq1\}$. Similarly, if $d$ is odd then $f$ is an odd map, so $C(f^{2})=C(T_{d^{2}})=\{\pm T_{e};\;e\geq1\}$. In both cases, we see that the description of the centralizers remains unchanged if we replace $f$ with an iterate, so $C(f^{\infty})=C(f)$ or $C(f^{2})$.

    \underline{Lattès case}. (Recall the notation from Subsection \ref{ss23}.) Suppose that $f=\ell_{\alpha,t}\in L(E,\Gamma)$ for some elliptic curve $E$ and group of automorphisms $\Gamma=\langle [\omega]\rangle$. Note that for $n\geq 1$, we have $f^{n}=\ell_{\alpha^{n},t_{n}}$, where $t_{n}=[1+\alpha+\cdots+\alpha^{n-1}](t)$.
    
    Let $g\in C(f^{\infty})$. By Lemma \ref{lem2} $(iii)$, $g$ is of the form $\ell_{\beta,s}\in L(E,\Gamma)$. For any $n\geq1$, $f^{n}$ and $g$ commmute if and only if the associated maps $\varphi,\psi:E\rightarrow E$ commute:
    \[\forall z\in E,\quad [\alpha^{n}]([\beta](z)+s)+t_{n}=[\beta]([\alpha^{n}](z)+t_{n})+s\]
    which is equivalent to
    \begin{equation}
        \label{eq7}
        [\alpha^{n}-1](s)=[\beta-1](t_{n}).
    \end{equation}
    Define $n=n(g)$ as the smallest integer such that $f^{n}$ commutes with $g$. Let $\tilde{\beta}$ be such that $\tilde{\beta}\equiv\beta\mod (1-\omega)$. Then Equation \eqref{eq7} still holds after replacing $\beta$ with $\tilde{\beta}$, because $t_{n}$ is a $[1-\omega]$-torsion point. Therefore, $\tilde{g}:=\ell_{\tilde{\beta},s}$ commutes with $f^{n}$, and thus $n(\tilde{g})\leq n(g)$. By exchanging the roles of $g$ and $\tilde{g}$, we have $n(\tilde{g})=n(g)$.
    
    We have shown that, for any $\ell_{\beta,s}\in C(f^{\infty})$, the integer $n(\ell_{\beta,s})$ only depends on the choice of $[1-\omega]$-torsion point $s$ and on the choice of $\beta$ modulo $(1-\omega)$. There are only finitely many such choices, and thus only finitely many $n(\ell_{\beta,s})$'s:
    \[\{n(\ell_{\beta,s})\;|\;\ell_{\beta,s}\in C(f^{\infty})\}=:\{n_{1},\cdots,n_{r}\}\]
    so by setting $N=n_{1}\cdots n_{r}$ we obtain $C(f^{\infty})=C(f^{N})$.
\end{proof}

\begin{example}
    Let $f:z\mapsto z^{2}$. For $n\geq1$, let $\zeta_{n}$ be a  primitive $(2^{n}-1)$-th root of unity. Then $g_{n}:z\mapsto \zeta_{n}z$ belongs to $C(f^{n})$ but not to $C(f^{k})$, $k\leq n$. This proves that $C(f^{\infty})$ is not $C(f^{N})$ for any $N\geq1$.
\end{example}

\subsection{The Tits alternative} Our proof of Theorem \ref{theo2} will revolve around embedding semigroups in groups. To that end, we will recall certain notions below. We refer the reader to \cite[§1.10]{CP61}, \cite[§0.8]{Coh85} for more details.\\

Let $S$ be a cancellative semigroup which does not contain a free subsemigroup of rank $2$. Then $S$ satifies the \textit{right and left Ore's conditions}:
\[\forall a,b\in S,\quad aS\cap bS\neq\varnothing,\;Sa\cap Sb\neq\varnothing.\]
This implies by Ore's theorem that $S$ embeds into its \textit{group of fractions}, which we will denote by $G_{S}$. In this context, $G_{S}$ can be defined explicitly as the set of $S$-valued pairs, written $ab^{-1}$, $a,b\in S$, modulo the equivalence relation
\[ab^{-1}\sim cd^{-1}\quad\Longleftrightarrow \quad \exists e,f\in S,\; ae=cf,\;be=df.\]
The group law in $G_{S}$ is then given by $(ab^{-1})(cd^{-1})=(au)(dv)^{-1}$, where $u,v\in S$ satisfy $bu=cv$. Note that any element $ab^{-1}\in G_{S}$ can also be written as $c^{-1}d$, where $ca=db$. The group of fractions satisfies the following universal property: any semigroup homomorphism $\varphi:S\rightarrow G$, where $G$ is a group, extends to a group homomorphism $\varphi:G_{S}\rightarrow G$. The extension is explicitly given by $\varphi(ab^{-1})=\varphi(a)\varphi(b)^{-1}$.

Let $T$ be a subsemigroup of $S$. The group generated by $T$ in $G_{S}$ is naturally isomorphic to $G_{T}$, so we get an inclusion $G_{T}<G_{S}$. We say that $T$ is of \textit{finite index} in $S$ if there exists a finite family $\{s_{1},\cdots, s_{r}\}\subset S$ such that for any $s\in S$, there is some $1\leq i\leq r$ such that $s_{i}s\in T$. In particular, we have the following inclusion in $G_{S}$:
\[S\subset\bigcup_{i=1}^{r}s_{i}^{-1}G_{T}.\]
For any $s\in S$, we have $sS\subset S$ so left multiplication by $s$ (and hence by $s^{-1}$) permutes the left cosets $s_{1}^{-1}G_{T},\cdots,s_{r}^{-1}G_{T}$. Since $S$ and $S^{-1}$ generate $G_{S}$, we obtain $G_{S}=s_{1}^{-1}G_{T}\cup\cdots\cup s_{r}^{-1}G_{T}$. In other words, $G_{T}$ is of finite index in $G_{S}$.\\

The other main ingredient in the proof of Theorem \ref{theo2} is the following criterion for finding free subsemigroups in $\End(\Ps^{1})^{+}$:

\begin{theorem}[{\cite[Theorem 1.3]{BHPT24}}]
    \label{theo7}
    Let $f,g$ be non-linear endomorphisms of $\Ps^{1}$ such that $\PrePer(f)\neq\PrePer(g)$. Then the semigroup generated by $f$ and $g$ contains a free subsemigroup of rank $2$.
\end{theorem}

\begin{proof}[Proof of Theorem \ref{theo2}]
    Let $S$ be a semigroup in $\End(\Ps^{1})$ defined over a finitely generated field $K$ of characteristic $0$. If $S$ satisfies a semigroup identity (assertion $(ii)$) or if its cancellative subsemigroups can be embedded in virtually abelian groups (assertion $(iii)$), then clearly $S$ cannot contain a free subsemigroup of rank $2$ (assertion $(i)$). Therefore, we only have to show that $(i)$ implies $(ii)$ and $(iii)$. Moreover, if $S$ is linear then these implications follow from \cite[Theorem 1]{OS95}. In light of this, we may assume for the rest of the proof that $(i)$ holds, and that $S$ contains a non-linear element $f$.

    By Theorem \ref{theo7}, for any non-linear $g\in S$ we must have $\PrePer(f)=\PrePer(g)$ and so $g(\PrePer(f))=\PrePer(f)$. For any linear $\sigma\in S$ we have $\PrePer(f)=\PrePer(\sigma f)$ and so $\sigma(\PrePer(f))=\PrePer(f)$. It follows that, for any $z\in\PrePer(f)$, the orbit $S\cdot z$ is contained in $\PrePer(f)$. Moreover, the orbit is defined over a finite extension of $K$: by the Northcott-Moriwaki property \cite[Theorem 4.3]{Mor00}, $S\cdot z$ is finite. We may choose a finite orbit $\mathcal{O}:=S\cdot z$ of size at least $3$ contained in the common Julia set $J$ of every non-linear element of $S$.

    Let $N\geq1$ be an integer such that, for any $\varphi:\mathcal{O}\rightarrow \mathcal{O}$, the $N$-th iterate $\varphi^{N}$ acts as the identity on $\varphi^{N}(\mathcal{O})$. We can choose for example $N=(\#\mathcal{O})!$. Let $g,h\in S$ be any pair of non-linear elements, and fix $z\in (g^{N}h^{N})^{N}(\mathcal{O})$. By our choice of $N$, $z$ is a fixed point of $(g^{N}h^{N})^{N}$. Similarly, $z\in g^{N}(\mathcal{O})$ so $z$ is also a fixed point of $g^{N}$. Since $g^{N}$ and $(g^{N}h^{N})^{N}$ have the same set of preperiodic points, they share the same measure of maximal entropy as a consequence of \cite[Theorem 1.4]{YZ21a}. We can therefore apply Lemma \ref{lem1}: $g^{N}$ and $(g^{N}h^{N})^{N}$ commute. We have obtained the following identity:
    \begin{equation}
        \label{eq8}
        g^{N}(g^{N}h^{N})^{N}=(g^{N}h^{N})^{N}g^{N}.
    \end{equation}
    Now, if any of the two endomorphisms is linear, say $g$, then $g^{N}$ has three fixed points in $\mathcal{O}$ and thus $g^{N}=\mathrm{id}$. In particular, Identity \eqref{eq8} still holds. This establishes $(i)\Rightarrow(ii)$.

    Let us now derive $(iii)$. Let $T$ be a cancellative subsemigroup of $S$. Since $T$ does not contain a free subsemigroup of rank $2$, it has a group of fractions $G_{T}$. If $T$ is linear, then $G_{T}$ is virtually abelian as a consequence of \cite[Theorem 1]{OS95}. We may therefore assume that $T$ contains a non-linear element $f$. We retain our choice of $\mathcal{O}$ and $N$ from our proof that $(i)\Rightarrow (ii)$. For any non-linear $g\in T$, we may apply Identity \eqref{eq8} to $f$ and $g$:
    \[f^{N}(f^{N}g^{N})^{N}=(f^{N}g^{N})^{N}f^{N}.\]
    By left cancellation, we obtain
    \[(f^{N}g^{N})^{N}=(g^{N}f^{N})^{N}.\]
    Let $z\in(f^{N}g^{N})^{N}(\mathcal{O})$. Then $z$ is contained in both $f^{N}(\mathcal{O})$ and $g^{N}(\mathcal{O})$, so it is a common fixed point of $f^{N}$ and $g^{N}$: by Lemma \ref{lem1}, $f^{N}$ and $g^{N}$ commute, and thus
    \begin{equation}
        \label{eq9}
        \Per(f)=\Per(g).
    \end{equation}
    Moreover, for any linear $\sigma\in T$ we have $\Per(f)=\Per(\sigma f)$, so $\sigma(\Per(f))=\Per(f)$. Applying once again the Northcott-Moriwaki property, we see that for any $w\in\Per(f)$, the orbit $T\cdot w\subset \Per(f)$ is finite.
    
    We may fix a finite orbit $\mathcal{O}_{0}:=T\cdot w$ that is contained in the Julia set of the non-linear elements of $T$. Define $T_{0}=\{g\in T\;|\; g_{|\mathcal{O}_{0}}=\mathrm{id}_{|\mathcal{O}_{0}}\}$, a subsemigroup of $T$. This subsemigroup is abelian: if $g,h$ are non-linear elements in $T_{0}$, then they commute by Lemma \ref{lem1}; if $\sigma\in T_{0}$ is linear and $g\in T_{0}$ is non-linear, then $g$ and $\sigma g$ commute so $g\sigma g=\sigma gg$ and hence $g\sigma =\sigma g$; if $\sigma,\tau\in T_{0}$ are both linear, then by the previous point $\sigma$ commutes with both $f$ and $\tau f$, so $\sigma\tau f=\tau f\sigma=\tau\sigma f$ and finally $\sigma\tau=\tau\sigma$. Thus, $G_{T_{0}}$ is abelian. The subsemigroup $T_{0}$ is also of finite index in $T$. Indeed, choose $\{h_{1},\cdots,h_{r}\}\subset T$ such that
    \[\{g_{|\mathcal{O}_{0}};\; g\in T\}=\{(h_{1})_{|\mathcal{O}_{0}},\cdots,(h_{r})_{|\mathcal{O}_{0}}\}.\]
    For any $g\in T$, letting $p\geq1$ denote the order of $g_{|\mathcal{O}_{0}}$, there exists some $1\leq i\leq r$ such that $(g^{p-1})_{|\mathcal{O}_{0}}=(h_{i})_{|\mathcal{O}_{0}}$, and so $h_{i}g\in T_{0}$. This proves that $G_{T_{0}}$ is of finite index in $G_{T}$, which means that $G_{T}$ is virtually abelian, establishing $(i)\Rightarrow(iii)$.
\end{proof}

\begin{proof}[{Proof of Corollary \ref{cor2}}]
    Let $S$ be a cancellative semigroup in $\End(\Ps^{1})^{+}$, and assume that $S$ does not contain a free subsemigroup of rank $2$. In particular, $S$ has a group of fractions $G_{S}$.
    
    Let us begin by showing that $G_{S}$ is locally virtually abelian. Let $H$ be a finitely generated subgroup of $G_{S}$. If we denote by $\{f_{1}g_{1}^{-1},\cdots,f_{r}g_{r}^{-1}\}$ a finite set of generators for $H$, then $H$ is contained in $G_{S_{0}}$, where $S_{0}$ is the subsemigroup of $S$ generated by $\{f_{1},g_{1},\cdots, f_{r},g_{r}\}$. The semigroup $S_{0}$ is defined over a finitely generated field and does not contain a free subsemigroup of rank $2$: by Theorem \ref{theo2}, $(i)\Rightarrow (iii)$, $G_{S_{0}}$ is virtually abelian and so is $H$.
    
    It only remains to show that $G_{S}$ is in fact virtually abelian, under the assumption that $S$ is not linear or a power semigroup.

    Let $f$ be a non-linear element of $S$. By Theorem \ref{theo7}, for any non-linear $g\in S$ we have $\PrePer(f)=\PrePer(g)$. If $f$ is conjugate to a Chebyshev polynomial (resp. to a Lattès map in $L(E,\Gamma)$), then by Lemma \ref{lem2}, every element of $S$ is simultaneously conjugate to a Chebyshev polynomial (resp. to a Lattès map in $L(E,\Gamma)$ for the same pair $(E,\Gamma)$). In both cases, $S$ is defined over a finitely generated field so we conclude by Theorem \ref{theo2}, $(i)\Rightarrow (iii)$. 

    Finally, assume that $f$ is non-exceptional. For any non-linear $g\in S$, the semigroup generated by $f$ and $g$ is defined over a finitely generated field. It is also cancellative and contains no free subsemigroup of rank $2$: by inspecting the proof of Theorem \ref{theo2}, $(i)\Rightarrow(iii)$, we see that $\Per(f)=\Per(g)$ (see \eqref{eq9} above). Similarly, if $\sigma\in S$ is linear, we have $\Per(f)=\Per(\sigma f)$ so $\sigma(\Per(f))=\Per(f)$. We have shown that $S\subset C(f^{\infty})$. By Theorem \ref{theo1}, $C(f^{\infty})$ is finitely generated so we may once again apply Theorem \ref{theo2}, $(i)\Rightarrow (iii)$ to conclude that $G_{S}$ is virtually abelian.
\end{proof}

\bibliographystyle{amsalpha}
\bibliography{bib}

@article{BHPT24,
    author = {J. P. Bell and K. Huang and W. Peng and T. J. Tucker},
    title = {A {T}its alternative for endomorphisms of the projective line},
    year = {2024},
    journal = {J. Eur. Math. Soc.},
    volume={26},
    number={12},
    pages={4903--4922}
}

@article{OS95,
    author = {J. Okniński and A. Salwa},
    title = {Generalised {T}its alternative for linear semigroups},
    journal = {J. Pure Appl. Algebra},
    volume={103},
    year = {1995},
    pages = {211--220},
}

@article{Mor00,
    author = {A. Moriwaki},
    title = {Arithmetic height functions over finitely generated fields},
    journal = {Invent. Math.},
    year =  {2000},
    volume={140},
    number={1},
    pages={101--141}
}

@misc{YZ21a,
    author = {X. Yuan and S. Zhang},
    title = {The arithmetic {H}odge index theorem for adelic line bundles {II}: finitely
generated fields},
    note = {arXiv preprint arXiv:1304.3539v2},
    year = {2021}
}

@article{Pak24,
    author = {F. Pakovich},
    title = {Semigroups of rational functions: some problems and conjectures},
    journal = {Panor. Synthèses},
    year = {2024},
    volume = {62},
    pages = {89--104}
}

@article{Pak20,
    author = {F. Pakovich},
    title = {Finiteness theorems for commuting and semiconjugate rational functions},
    journal = {Conform. Geom. Dyn.},
    year = {2020},
    volume = {24},
    pages = {202--229}
}

@article{Ere90,
    author = {A. Eremenko},
    title = {Some functional equations connected with the iteration of rational functions},
    journal = {Leningrad Math. J.},
    year = {1990},
    volume = {1},
    pages = {905--919}
}

@article{Rit23,
    author = {J. F. Ritt},
    title = {Permutable rational functions},
    journal = {Trans. Amer. Math. Soc.},
    year = {1923},
    volume = {25},
    pages = {399--448}
}

@incollection{Mil06,
  author      = {J. Milnor},
  title       = {On {L}attès maps},
  editor      = {P. G. Hjorth and C. L. Petersen},
  booktitle   = {Dynamics on the Riemann Sphere},
  publisher   = {EMS Press},
  address     = {Zurich},
  year        = {2006},
  pages       = {9--43}
}

@article{Lev90,
    author = {G. Levin},
    title = {Symmetries on {J}ulia sets},
    journal = {Math. Notes},
    year = {1990},
    volume = {48},
    number = {5--6},
    pages = {1126--1131}
}

@article{LP97,
    author = {G. Levin and F. Przytycki},
    title = {When do two rational functions have the same {J}ulia set?},
    journal = {Proc. Amer. Math. Soc.},
    year = {1997},
    volume = {125},
    number = {7},
    pages = {2179--2190}
}

@article{Lju83,
    author = {M. Ljubich},
    title = {Entropy properties of rational endomorphisms of the {R}iemann sphere},
    journal = {Ergodic Theory Dynam. Systems},
    year = {1983},
    volume = {3},
    number = {3},
    pages = {351--385}
}

@article{FLM83,
    author = {A. Freire and A. Lopes and R. Mañé},
    title = {An invariant measure for rational maps},
    journal = {Bol. Soc. Brasil. Mat.},
    year = {1983},
    volume = {14},
    number = {1},
    pages = {45--62}
}

@article{BE87,
    author = {I. N. Baker and A. Eremenko},
    title = {A problem on {J}ulia sets},
    journal = {Ann. Acad. Sci. Fenn.},
    year = {1987},
    volume = {12},
    pages = {229--236}
}

@article{SS95,
  title={The polynomials associated with a {J}ulia set},
  author={W. Schmidt and N. Steinmetz},
  journal={Bull. Lond. Math. Soc.},
  volume={27},
  number={126},
  pages={239--241},
  year={1995},
}

@article{Hug23,
  title={Rational maps with rational multipliers},
  author={V. Huguin},
  journal={J. Ec. Polytech. - Math.},
  volume={10},
  pages={591--599},
  year={2023},
}

@article{Ye15,
  title={Rational functions with identical measure of maximal entropy},
  author={H. Ye},
  journal={Adv. Math.},
  volume={268},
  pages={373--395},
  year={2015},
}

@book{BR10,
    author = {M. Baker and R. Rumely},
    title = {Potential theory and dynamics on the {B}erkovich projective line},
    publisher = {American Mathematical Society},
    series = {Mathematical Surveys and Monographs},
    volume = {159},
    year = {2010},
    address = {Providence, RI},
}

@article{Zdu14,
    title={Characteristic exponents of rational   functions},
    author={A. Zdunik},
    journal={Bull. Pol. Acad. Sci. Math.},
    volume={62},
    number = {3},
    pages={257--263},
    year={2014},
}

@article{KS07,
    title={Dynamics of projective morphisms having identical canonical heights},
    author={S. Kawaguchi and J. H. Silverman},
    journal={Proc. London Math. Soc.},
    volume={95},
    number = {3},
    pages={519--544},
    year={2007},
}

@book{Bea91,
    title = {Iteration of Rational Functions},
    author = {A. F. Beardon},
    series = {Grad. Texts in Math.},
    volume={132},
    year = {1991},
    publisher = {Springer},
    address = {New York}
}

@book{CP61,
    author = {A. H. Clifford and G. B. Preston},
    title = {The algebraic theory of semigroups, volume 1},
    publisher = {American Mathematical Society},
    series = {Mathematical Surveys and Monographs},
    volume = {7},
    year = {1961},
    address = {Providence, RI},
}

@book{Coh85,
    author = {P. M. Cohn},
    title = {Free rings and their relations},
    publisher = {Academic Press},
    year = {1985},
    address = {London},
}

@article{Sil98,
    author = {J. H. Silverman},
    title = {The space of rational maps on $\mathbb{P}^{1}$},
    journal = {Duke Math. J.},
    year = {1998},
    volume = {94},
    number = {1},
    pages = {41--77}
}

@misc{DFR23,
    title={On the dynamical Manin-Mumford conjecture for plane polynomial maps}, 
    author={R. Dujardin and C. Favre and M. Ruggiero},
    year={2023},
    note = {arXiv preprint arXiv:2312.14817},
}

@article{DF17,
    author = {R. Dujardin and C. Favre},
    title = {The dynamical Manin-Mumford problem for plane polynomial automorphisms},
    journal = {J. Eur. Math. Soc.},
    year = {2017},
    volume = {19},
    number = {11},
    pages = {3421--3465}
}

@article{Fat23,
    author = {P. Fatou},
    title = {Sur l'itération analytique et les substitutions permutables},
    journal = {J. Math. Pures Appl.},
    volume = {2},
    year = {1923},
    pages = {343--384}
}

@article{Jul22,
    author = {G. Julia},
    title = {Mémoire sur la permutabilité des fractions rationnelles},
    journal = {Ann. Sci. Éc. Norm. Supér.},
    volume = {39},
    year = {1922},
    pages = {131--215}
}

@article{Bel23,
    author = {J. Bellaïche},
    title = {On self-correspondences on curves},
    journal = {Algebra and Number Theory},
    volume = {17},
    number = {11},
    year = {2023},
    pages = {1867--1899}
}

@article{DS02,
    author = {T. C. Dinh and N. Sibony},
    title = {Sur les endomorphismes holomorphes permutables de $\mathbb{P}^{k}$},
    journal = {Math. Ann.},
    volume = {324},
    number = {1},
    year = {2002},
    pages = {33--70}
}

@article{Kau18,
    author = {L. Kaufmann},
    title = {Commuting pairs of endomorphisms of $\mathbb{P}^{2}$},
    journal = {Ergod. Theory Dyn. Syst.},
    volume = {38},
    number = {3},
    year = {2018},
    pages = {1025--1047}
}

@article{Aut01,
    author = {P. Autissier},
    title = {Points entiers sur les surfaces arithmétiques},
    journal = {J. Reine Angew. Math.},
    volume = {531},
    year = {2001},
    pages = {201--235}
}

@article{Mañ83,
    author = {R. Ma{\~n}\'e},
    title = {On the uniqueness of the maximizing measure for rational maps},
    journal = {Bol. Soc. Bras. Math.},
    year = {1983},
    volume = {14},
    number = {1},
    pages = {27--43}
}

@book{Mil06b,
    author = {J. Milnor},
    title = {Dynamics in one complex variable},
    publisher = {Princeton University Press},
    year = {2006},
    address = {Princeton},
}

@article{Bro65,
    author = {H. Brolin},
    title = {Invariant sets under iteration of rational functions},
    journal = {Ark. Mat.},
    year = {1965},
    volume = {6},
    number = {2},
    pages = {103--144}
}

@article{Lev01,
    author = {G. Levin},
    title = {Letter to the editor},
    journal = {Math. Notes},
    year = {2001},
    volume = {69},
    number={3},
    pages = {432--433}
}

@article{Pak21,
    author = {F. Pakovich},
    title = {Commuting rational functions revisited},
    journal = {Ergodic Theory Dynam. Systems},
    year = {2021},
    volume = {41},
    number = {1},
    pages = {295--320}
}

@article{Pak22b,
    author = {F. Pakovich},
    title = {On amenable semigroups of rational functions},
    journal = {Trans. Amer. Math. Soc.},
    year = {2022},
    volume = {375},
    number = {11},
    pages = {7945--7979}
}

@article{CM21,
    author = {C. Cabrera and P. Makienko},
    title = {Amenability and measure of maximal entropy for semigroups of rational maps},
    journal = {Groups Geom. Dyn.},
    year = {2021},
    volume = {15},
    number = {4},
    pages = {1139--1174}
}

@article{CM23,
    author = {C. Cabrera and P. Makienko},
    title = {On amenability and measure of maximal entropy for semigroups of rational maps: {II}},
    journal = {Internat. J. Algebra Comput.},
    year = {2023},
    volume = {33},
    number = {6},
    pages = {1099--1125}
}

\end{document}